\documentclass[acmsmall]{acmart}

\AtBeginDocument{%
  \providecommand\BibTeX{{%
    \normalfont B\kern-0.5em{\scshape i\kern-0.25em b}\kern-0.8em\TeX}}}

\acmJournal{TECS}
\usepackage{algpseudocode}
\usepackage{amsmath}
\usepackage{amsthm}
\usepackage{booktabs}
\usepackage{cuted}
\usepackage{enumerate}
\usepackage{listliketab}
\usepackage{mathtools}
\usepackage{mfirstuc}
\usepackage{paralist}
\usepackage{resizegather}
\usepackage{soul}
\usepackage{url}
\usepackage{xcolor}
\usepackage{numprint}
\usepackage{caption}
\usepackage{subcaption}

\usepackage[capitalise]{cleveref}
\crefname{definition}{Definition}{Definitions}
\crefname{theorem}{Theorem}{Theorems}
\crefname{lemma}{Lemma}{Lemmas}
\crefname{corollary}{Corollary}{Corollaries}
\crefname{algocf}{Algorithm}{Algorithms}

\usepackage[ruled]{algorithm2e} 

\SetAlFnt{\small}
\SetAlCapFnt{\small}
\SetAlCapNameFnt{\small}
\SetAlCapHSkip{0pt}
\IncMargin{-\parindent}
\SetKwProg{Function}{Function}{}{}
\definecolor{DarkGreen}{RGB}{0,100,0}

\SetCommentSty{CommentFont}

\usepackage{tikz}
\usepackage{flowchart}
\usetikzlibrary{matrix, shapes.geometric, arrows, calc, patterns, math}

\usepackage{xifthen}
\usepackage{xparse}

\definecolor{redish}      {rgb}{0.8, 0.1, 0.1}
\definecolor{blueish}     {rgb}{0  , 0  , 0.4}
\definecolor{greenish}    {rgb}{0  , 0.6, 0  }
\definecolor{yellowish}   {rgb}{0.8, 0.5, 0  }
\definecolor{redishBG}    {rgb}{1  , 0.3, 0.3}
\definecolor{blueishBG}   {rgb}{0.4, 0.4, 1  }
\definecolor{greenishBG}  {rgb}{0.2, 1  , 0.2}
\definecolor{yellowishBG} {rgb}{1  , 0.7, 0  }

\newcommand{\Alg}   {Algorithm}
\newcommand{\Apx}   {Appendix}
\newcommand{\Cor}   {Corollary}
\newcommand{\Def}   {Definition}
\newcommand{\Eq}    {Equation}
\newcommand{\Fig}   {Figure}
\newcommand{\Lem}   {Lemma}
\newcommand{\Prop}  {Proposition}
\newcommand{\Rmrk}  {Remark}
\newcommand{\Sec}   {Section}
\newcommand{\Tbl}   {Table}
\newcommand{\Thm}   {Theorem}
\newcommand{\Cnd}   {Condition}
\newcommand{\Prob}  {Problem}
\newcommand{\Ex}    {Example}
\newcommand{\Itm}   {Item}
\newcommand{\Typ}   {Type}

\newtheorem{remark}{\Rmrk}

\NewDocumentCommand{\aref}{m}{%
  \IfBeginWith{#1}{sec:} {\Sec~\ref{#1}}{%
  \IfBeginWith{#1}{thm:} {\Thm~\ref{#1}}{%
  \IfBeginWith{#1}{lem:} {\Lem~\ref{#1}}{%
  \IfBeginWith{#1}{cor:} {\Cor~\ref{#1}}{%
  \IfBeginWith{#1}{def:} {\Def~\ref{#1}}{%
  \IfBeginWith{#1}{apx:} {\Apx~\ref{#1}}{%
  \IfBeginWith{#1}{alg:} {\Alg~\ref{#1}}{%
  \IfBeginWith{#1}{fig:} {\Fig~\ref{#1}}{%
  \IfBeginWith{#1}{tbl:} {\Tbl~\ref{#1}}{%
  \IfBeginWith{#1}{rem:} {\Rem~\ref{#1}}{%
  \IfBeginWith{#1}{cnd:} {\Cnd~\ref{#1}}{%
  \IfBeginWith{#1}{prb:} {\Prob~\ref{#1}}{%
  \IfBeginWith{#1}{itm:} {\Itm~\ref{#1}}{%
  \IfBeginWith{#1}{eq:}  {\Eq~\ref{#1}}{%
  \IfBeginWith{#1}{ex:}  {\Ex~\ref{#1}}{%
  \IfBeginWith{#1}{typ:} {\Typ~\ref{#1}}{%
  \IfBeginWith{#1}{prop:}{\Prop~\ref{#1}}{%
  \errmessage{class of label '#1' is not defined.}%
  }}}}}}}}}}}}}}}}}%
}

\newenvironment{myitems}{\begin{itemize}}{\end{itemize}}

\DeclarePairedDelimiter{\Ceil}  {\lceil}    {\rceil}
\DeclarePairedDelimiter{\Floor} {\lfloor}   {\rfloor}
\DeclarePairedDelimiter{\Size}  {\vert}     {\vert}
\DeclarePairedDelimiter{\Paren} {\lparen}   {\rparen}
\DeclarePairedDelimiter{\Brace} {\lbrace}   {\rbrace}
\DeclarePairedDelimiter{\Braket}{\lbrack}   {\rbrack}
\DeclarePairedDelimiter{\Sem}   {\llbracket}{\rrbracket}

\DeclarePairedDelimiter{\LOne}   {\lVert}{\rVert_1}

\DeclarePairedDelimiter{\LInf}   {\lVert}{\rVert_\infty} 

\NewDocumentCommand{\GEz}     {}{{\scalebox{0.5}{${\geq}0$}}}

\NewDocumentCommand{\Nat}     {}{\ensuremath{\mathbb{N}}}

\NewDocumentCommand{\Rat}     {}{\ensuremath{\mathbb{Q}}}
\NewDocumentCommand{\Real}    {}{\ensuremath{\mathbb{R}}}

\NewDocumentCommand{\nnRat}   {}{\ensuremath{\Rat_\GEz}}
\NewDocumentCommand{\nnReal}  {}{\ensuremath{\Real_\GEz}}

\NewDocumentCommand{\iFF} {}{\mbox{iff}}    

\NewDocumentCommand{\ie}  {}{{\em i.e.}}

\NewDocumentCommand{\wLOG}{}{{\em wlog.}}

\NewDocumentCommand{\oftype}  {}{\ensuremath{\in}}

\NewDocumentCommand{\Sat}     {}{\ensuremath{\models}}
\NewDocumentCommand{\nSat}    {}{\ensuremath{\not\models}}

\NewDocumentCommand{\goesto}      {O{}}   {\ensuremath{\xrightarrow{#1}}}
\NewDocumentCommand{\goestoo}     {O{}O{}}{\ensuremath{\underset{#2}{\goesto[#1]}}}
\NewDocumentCommand{\goesboo}     {O{}}   {\ensuremath{\underset{#1}{\goesto}}}

\NewDocumentCommand{\hQ}    {}{\ensuremath{\mathtt{Q}}}
\NewDocumentCommand{\hE}    {}{\ensuremath{\mathtt{E}}}
\NewDocumentCommand{\hS}    {}{\ensuremath{\mathtt{S}}}
\NewDocumentCommand{\hD}    {}{\ensuremath{\mathtt{D}}}
\NewDocumentCommand{\hX}    {}{\ensuremath{\mathtt{X}}}
\NewDocumentCommand{\hI}    {}{\ensuremath{\mathtt{I}}}
\NewDocumentCommand{\hJ}    {}{\ensuremath{\mathtt{J}}}

\NewDocumentCommand{\hA}    {}{\ensuremath{\Sigma}}
\NewDocumentCommand{\hL}    {}{\ensuremath{\mathtt{L}}}
\NewDocumentCommand{\hQi}   {}{\ensuremath{\hQ^\mathsf{init}}}
\NewDocumentCommand{\hQf}   {}{\ensuremath{\hQ^\mathtt{final}}}

\NewDocumentCommand{\bS}   {}{\ensuremath{\mathtt{S}}}
\NewDocumentCommand{\bSi}  {}{\ensuremath{\bS^\mathsf{init}}}

\NewDocumentCommand{\bA}   {}{\ensuremath{\Sigma}}
\NewDocumentCommand{\bT}   {}{\ensuremath{\Gamma}}
\NewDocumentCommand{\bF}   {}{\ensuremath{\mathtt{F}}}

\NewDocumentCommand{\AP}     {}{\ensuremath{\mathtt{AP}}}

\NewDocumentCommand{\Yes}    {}{\ensuremath{\mathtt{yes}}}
\NewDocumentCommand{\No}     {}{\ensuremath{\mathtt{no}}}
\NewDocumentCommand{\Unknown}{}{\ensuremath{\mathtt{unknown}}}

\NewDocumentCommand{\MITL}     {}{MITL}
\NewDocumentCommand{\LTL}      {}{LTL}

\NewDocumentCommand{\PSPACE}          {}{\mbox{\scalebox{0.85}{\textsf{PSPACE}}}}

\NewDocumentCommand{\EXPSPACE}        {}{\mbox{\scalebox{0.85}{\textsf{EXPSPACE}}}}

\NewDocumentCommand{\PSPACEcomp}      {}{\PSPACE-complete}

\NewDocumentCommand{\EXPSPACEcomp}    {}{\EXPSPACE-complete}

\NewDocumentCommand{\Term}{mmm}{%
  \expandafter\DeclareExpandableDocumentCommand\expandafter{\csname #1ls\endcsname}{}{\MakeLowercase{#2}}%
  \expandafter\DeclareExpandableDocumentCommand\expandafter{\csname #1lp\endcsname}{}{\MakeLowercase{#3}}%
  \expandafter\DeclareExpandableDocumentCommand\expandafter{\csname #1cs\endcsname}{}{#2}%
  \expandafter\DeclareExpandableDocumentCommand\expandafter{\csname #1cp\endcsname}{}{#3}%
  \expandafter\DeclareExpandableDocumentCommand\expandafter{\csname #1us\endcsname}{}{\makefirstuc{\MakeLowercase{#2}}}%
  \expandafter\DeclareExpandableDocumentCommand\expandafter{\csname #1up\endcsname}{}{\makefirstuc{\MakeLowercase{#3}}}%
}

\NewDocumentCommand{\STORMED}{}{\mbox{\scalebox{0.9}{STORMED}}}
\NewDocumentCommand{\Buchi}  {}{B\"uchi}
\NewDocumentCommand{\Markov} {}{Markov}

\Term{BA}       {\Buchi\ Automaton}                           {\Buchi\ Automata}
\Term{GBA}      {Generalized \Buchi\ Automaton}               {Generalized \Buchi\ Automata}
\Term{TS}       {Transition System}                           {Transition Systems}
\Term{HS}       {Hybrid System}													      {Hybrid Systems}
\Term{HA}       {Hybrid Automaton}                            {Hybrid Automata}
\Term{TA}       {Timed Automaton}                             {Timed Automata}
\Term{rTA}      {Rational Timed Automaton}                    {Rational Timed Automata}
\Term{irTA}     {Irrational Timed Automaton}                  {Irrational Timed Automata}
\Term{PolyHA}   {Polyhedral Automaton}                        {Polyhedral Automata}
\Term{nlfPolyHA}{Non-linear Polyhedral Automaton}             {Non-linear Polyhedral Automata}
\Term{RHA}      {Rectangular Automaton}                       {Rectangular Automata}
\Term{MRHA}     {Monotonic Rectangular Automaton}             {Monotonic Rectangular Automata}
\Term{iMRHA}    {Initialzied Monotonic Rectangular Automaton} {Initialized Monotonic Rectangular Automata}
\Term{LIHA}     {Linear Inclusion  Automaton}                 {Linear Inclusion Automata}
\Term{iRHA}     {Initialized Rectangular Automaton}           {Initialized Rectangular Automata}
\Term{iLIHA}    {Initialized Linear Inclusion Automaton}      {Initialized Linear Inclusion Automata}
\Term{AHA}      {Affine Hybrid Automaton}                     {Affine Hybrid Automata}
\Term{LEHA}     {Linear Equation Automaton}                   {Linear Equation Automata}
\Term{iLEHA}    {Initialized Linear Equation Automaton}       {Initialized Linear Equation Automata}
\Term{SWHA}     {Stopwatch Automaton}                         {Stopwatch Automata}
\Term{iSWHA}    {Initialized Stopwatch Automaton}             {Initialized Stopwatch Automata}
\Term{SHA}      {Solvable Hybrid Automaton}                   {Solvable Hybrid Automata}
\Term{SAOS}     {Semi-Algebraic O-Minimal System}             {Semi-Algebraic O-Minimal Systems}
\Term{SASS}			{Semi-Algebraic \STORMED\ System}			        {Semi-Algebraic \STORMED\ Systems}
\Term{tCM}			{$2$-Counter Machine}									        {$2$-Counter Machines}
\Term{DTMC}     {Discrete Time \Markov\ Chain}                {Discrete Time \Markov\ Chains}
\Term{CTMC}     {Continuous Time \Markov\ Chain}              {Continuous Time \Markov\ Chains}
\Term{PMF}      {Probability Mass Function}                   {Probability Mass Functions}
\Term{LTL}      {Linear Temporal Logic}                       {Linear Temporal Logics}
\Term{iLTL}     {Linear Inequality \LTL}                      {Linear Inequality \LTL}
\Term{MTL}      {Metric Temporal Logic}                       {Metric Temporal Logics}
\Term{MITL}     {Metric Interval Temporal Logic}              {Metric Interval Temporal Logics}
\Term{NNF}      {Negation Normal Form}                        {Negation Normal Forms}
\Term{TS}       {Transition System}                           {Transition Systems}
\Term{STL}      {Signal Temporal Logic}                       {Signal Temporal Logics}

\NewDocumentCommand{\UnaryFunc}{mmo}{\ensuremath{%
  \IfValueTF{#3}%
  {\IfBooleanTF{#1}%
    {#2\Paren*{#3}}%
    {#2\Paren{#3}}%
  }%
  {#2}%
}}

\NewDocumentCommand{\BigO}  {so}{\UnaryFunc{#1}{\mathcal{O}}[#2]}
\NewDocumentCommand{\Inf}   {so}{\UnaryFunc{#1}{\mathtt{inf}}[#2]}
\NewDocumentCommand{\iInf}  {so}{\UnaryFunc{#1}{\mathtt{inf}}[#2]}
\NewDocumentCommand{\iSup}  {so}{\UnaryFunc{#1}{\mathtt{sup}}[#2]}
\NewDocumentCommand{\Lang}  {so}{\UnaryFunc{#1}{\mathtt{Lang}}[#2]}

\NewDocumentCommand{\Func}{mmoO{\Paren*}}
{%
    \IfNoValueTF{#3}%
    {\ensuremath{#1#2}}%
    {\ensuremath{#1#2#4{#3}}}%
}

\NewDocumentCommand{\Finally}   {oO{}}  {\Func{\mathtt{F}}{#2}[#1][]}
\NewDocumentCommand{\Prb}       {oO{}}  {\Func{\mathbb{P}}{#2}[#1][\Braket*]}
\NewDocumentCommand{\Exp}       {oO{}}  {\Func{\mathbb{E}}{#2}[#1]}
\NewDocumentCommand{\Path}      {oO{}}  {\Func{\mathtt{Path}}{#2}[#1]}
\NewDocumentCommand{\IEq}       {oO{}}  {\Func{\mathtt{IEq}}{#2}[#1]}
\NewDocumentCommand{\Vesta}     {oO{}}  {\Func{\mathtt{VeStA}}{#2}[#1]}
\NewDocumentCommand{\Neg}       {oO{}}  {\Func{\mathtt{neg}}{#2}[#1]}
\NewDocumentCommand{\CL}        {oO{}}  {\Func{\mathtt{cl}}{#2}[#1]}
\NewDocumentCommand{\CS}        {oO{}}  {\Func{\mathtt{cs}}{#2}[#1]}
\NewDocumentCommand{\Atoms}     {oO{}}  {\Func{\mathtt{atoms}}{#2}[#1]}
\NewDocumentCommand{\ALG}       {oO{}}  {\Func{\mathcal{A}}{#2}[#1]}
\NewDocumentCommand{\ALGClose}  {oO{}}  {\Func{\mathtt{Close}}{#2}[#1]}

\DeclareExpandableDocumentCommand{\UntilOp}			{}{\ensuremath{\mathcal{U}}}
\DeclareExpandableDocumentCommand{\ReleaseOp}		{}{\ensuremath{\mathcal{R}}}
\DeclareExpandableDocumentCommand{\NextOp}			{}{\ensuremath{\mathcal {X}}}
\DeclareExpandableDocumentCommand{\AlwaysOp}		{}{\ensuremath{\square}}
\DeclareExpandableDocumentCommand{\EventuallyOp}{}{\ensuremath{\lozenge}}
\NewDocumentCommand{\Next}			{mo} {\ensuremath{\NextOp\IfValueT{#2}{_{#2}}#1}}
\NewDocumentCommand{\Always}		{mo} {\ensuremath{\AlwaysOp\IfValueT{#2}{_{#2}}#1}}
\NewDocumentCommand{\Eventually}{mo} {\ensuremath{\EventuallyOp\IfValueT{#2}{_{#2}}#1}}
\RenewDocumentCommand{\Until}		{mmo}{\ensuremath{#1\UntilOp\IfValueT{#3}{_{#3}}#2}}
\NewDocumentCommand{\Release}		{mmo}{\ensuremath{#1\ReleaseOp\IfValueT{#3}{_{#3}}#2}}

\NewDocumentCommand{\mR}    {}{\ensuremath{{r}}}

\NewDocumentCommand{\mY}    {}{\ensuremath{{p}}}
\NewDocumentCommand{\mC}    {}{\ensuremath{{c}}}

\NewDocumentCommand{\Restriction}{o}
{%
    \IfNoValueTF{#1}%
    {\textcolor{red}{(restriction)}}%
    {\textcolor{red}{(restriction: #1)}}%
}

\newcommand{\mc}[1]{\ensuremath{\mathcal{#1}}}
\newcommand{\mbf}[1]{\ensuremath{\mathbf{#1}}}

\renewcommand{\d}{\ensuremath{\mathrm{d}}}
\NewDocumentCommand{\od}{omm}{\ensuremath{%
\IfNoValueTF{#1}{\frac{\mathrm{d} #2}{\mathrm{d} #3}}{\frac{\mathrm{d}^{#1} #2}{\mathrm{d} {#3}^{#1}}}}}

\newcommand{\real}{\ensuremath{\mathbb{R}}}

\newcommand{\uni}{\ensuremath{\mathbf{U}}}

\renewcommand{\subset}{\subseteq}
\renewcommand{\phi}{\varphi}
\newcommand{\paren}[1]{\ensuremath{\left( #1 \right) }}

\newcommand{\set}[2]{\ensuremath{\left\{ #1 \mid #2 \right\}}}

\newcommand{\abs}[1]{\ensuremath{\vert #1 \vert}}
\newcommand{\nm}[1]{\ensuremath{\Vert #1 \Vert}}

\newcommand{\tv}[1]{\nm{#1}_{\text{TV}}}

\newcommand{\ner}{\Delta}
\newcommand{\inv}[1]{{#1}^{\rm inv}}
\renewcommand{\d}{\ensuremath{\mathrm{d}}}

\begin{document}

\title{Verifying Stochastic Hybrid Systems with Temporal Logic Specifications via Model Reduction}

\author{Yu Wang}
\orcid{0000-0002-0431-1039}
\affiliation{%
  \institution{Duke University}
  \streetaddress{100 Science Dr}
  \city{Durham}
  \state{NC}
  \postcode{27708}
}
\email{yu.wang094@duke.edu}

\author{Nima Roohi}
\affiliation{%
  \institution{University of California San Diego}
  \streetaddress{9500 Gilman Dr}
  \city{La Jolla}
  \state{CA}
  \postcode{92093}
}
\email{nroohi@ucsd.edu}

\author{Matthew West}
\affiliation{%
  \institution{University of Illinois at Urbana-Champaign}
  \streetaddress{1206 W Green St}
  \city{Urbana}
  \state{IL}
  \postcode{61801}
}
\email{mwest@illinois.edu}

\author{Mahesh Viswanathan}
\affiliation{%
  \institution{University of Illinois at Urbana-Champaign}
  \streetaddress{201 Goodwin Ave}
  \city{Urbana}
  \state{IL}
  \postcode{61801}
}
\email{vmahesh@illinois.edu}

\author{Geir E. Dullerud}
\affiliation{%
  \institution{University of Illinois at Urbana-Champaign}
  \streetaddress{1308 W Main St}
  \city{Urbana}
  \state{IL}
  \postcode{61801}
}
\email{dullerud@illinois.edu}

\begin{abstract}
We present a scalable methodology to verify stochastic hybrid systems. Using the Mori-Zwanzig reduction method, we construct a finite state Markov chain reduction of a given stochastic hybrid system and prove that this reduced Markov chain is approximately equivalent to the original system in a distributional sense. Approximate equivalence of the stochastic hybrid system and its Markov chain reduction means that analyzing the Markov chain with respect to a suitably strengthened property, allows us to conclude whether the original stochastic hybrid system meets its temporal logic specifications. We present the first statistical model checking algorithms to verify stochastic hybrid systems against correctness properties, expressed in the linear inequality linear temporal logic (iLTL) or the metric interval temporal logic (MITL).
\end{abstract}



\maketitle

\section{Introduction} \label{sec:intro}

Stochastic hybrid systems, modeling discrete, continuous, and stochastic behavior, arise in many real-world applications ranging from automobiles~\cite{jin2014benchmarks}, smart grids~\cite{daniele2017smart}, and
biology~\cite{rajkumar2010cyber,liu_probabilistic_2011,liu_approximate_2012,zuliani_statistical_2014,gyori2015approximate}. 
In these contexts, it is often useful to determine if the models meet their time-dependent design goals. However, the verification problem is computationally very challenging --- even for systems with very simple dynamics that exhibit no stochasticity, and for the most basic class of safety properties, namely invariants, the problem of determining if a system meets its safety goals is undecidable~\cite{decidable-hybrid98}. 
The difficulty of the verification problem largely arises from the fact that the state space of such systems has uncountably many states.

The computational challenge posed by the verification problem is often addressed by constructing a simpler finite state model of the system, and then analyzing the finite state model. 
The finite state model is typically an \emph{abstraction} or a conservative over-approximation of the original system, \ie, every behavior of the system is exhibited by the finite state model, but the finite state model may have additional behaviors that are not system behaviors. 
This approach has been used to verify~\cite{efhkost03-2,adi03,rpv17} and design controllers~\cite{tabuada_linear_2006,kloetzer_fully_2008,wongpiromsarn_receding_2010,liu2013synthesis} for non-stochastic systems, as well as to verify~\cite{cv09-2,liu_probabilistic_2011,liu_approximate_2012,tkachev2013formula,gyori2015approximate} and design controllers~\cite{tkachev2013quantitative} for stochastic hybrid systems. 
For such abstractions, if the finite state model is safe then so is the original system. However, if the finite state model is unsafe, then not much can be concluded about the safety of the original system because the finite state model is an over-approximation.

In this paper, we present a scalable approach to the verification of a class of specifications defined by iLTL or MITL~\cite{04-iLTL,96-MITL} for stochastic hybrid systems, based on constructing a finite state approximation that is ``equivalent'' to the original system.
These specifications reason over the evolution of the probability distributions of the systems, and can express a wide class of safety properties.

To verify these specifications, we construct an approximate bi-simulation between stochastic hybrid systems and finite state Markov chains using the Mori-Zwanzig model reduction method~\cite{chorin_optimal_2000,beck_model_2009}.
The advantage of bi-simulation is that analyzing the finite state model not only allows us to conclude the safety of the hybrid stochastic system, but also its non-safety. 
In order to explain the relationship between the Markov chain we construct and the stochastic hybrid system, it is useful to recall that there are two broad approaches to defining the semantics of a stochastic process. 
One approach is to view a stochastic system as defining a measure space on the collection of executions; by execution here we mean a sequence of states that the system may possibly go through. 
The other approach is to view the stochastic system as defining a transformation on distributions; in such a view, the behavior of the stochastic model is captured by a sequence of distributions, starting from some initial distribution. 
For the first semantics (of measures on executions), it has been shown that approximate abstractions can be build 
between finite state Markov chains and certain classes of stochastic hybrid systems~\cite{julius_ApproximationsStochasticHybrid_2009,abate_ApproximateModelChecking_2010a,abate_ApproximateAbstractionsStochastic_2011}.
However, it has been observed that constructing an approximate ``equivalence'' between Markov chains and infinite-state systems is very challenging in general~\cite{gyori2015approximate}.

In this paper, we in contrast show that the Mori-Zwanzig reduction method constructs a finite state Markov chain that is approximately equivalent to a stochastic hybrid system with respect to the second semantics. 
That is, we show that the distribution on states of the Markov chain at any time, is close to the distribution at the same time defined by the stochastic hybrid system (Theorem~\ref{thm:reduction error}), even though there might be no (approximate) probabilistic path-to-path correspondence between the path space of the stochastic hybrid system and that of the Markov chain, as it is required under the first semantics.

Similar to~\cite{abate_ApproximateModelChecking_2010a,abate_ApproximateAbstractionsStochastic_2011}, the Mori-Zwanzig reduction is performed via partitioning the state space, although the metric for ``equivalence'' is different.
The approximate equivalence by Mori-Zwanzig reduction can be seen to be similar in spirit to the results first established for non-stochastic, stable, hybrid systems~\cite{tabuada_linear_2006,Pola20082508,girard2010approximately}, and later extended to stochastic dynamical systems~\cite{zamani2012symbolic,zamani2014symbolic}. 
When compared to~\cite{zamani2012symbolic,zamani2014symbolic}, we consider a more general class of stochastic hybrid systems that have multiple modes and jumps with guards and resets. 
Second, our reduced system is a Markov chain, whereas in~\cite{zamani2012symbolic,zamani2014symbolic} the stochastic system is approximated by a finite state, non-stochastic model.
In addition, our notion of distance between the stochastic hybrid system and the reduced system is slightly different.

Having proved that our reduced Markov model is approximately equivalent to the original stochastic hybrid system, we can exploit this to verify stochastic hybrid systems. 
Approximate equivalence ensures that analyzing the reduced model with respect to a suitably strengthened property, allows us to determine whether the initial stochastic hybrid system meets or violates its requirements. 
Therefore, a scalable verification approach can be obtained by developing algorithms to verify finite state Markov
chains. 
Since the reduced system, even though finite state, is likely to have a large number of states, we use a statistical approach to verification~\cite{younes_statistical_2006} as opposed to a symbolic one.

In statistical model checking, the model being verified is simulated multiple times, and the drawn simulations are analyzed to see if they constitute a statistical evidence for the correctness of the model. 
Statistical model checking algorithms have been developed for logics that reason about measures of executions~\cite{younes_statistical_2006,sen_statistical_2005,yesno-to-yesnounknown-2006-Younes,zuliani_statistical_2014}. 
However, since our reduced Markov chain is only close to the stochastic hybrid system in a distributional sense, we cannot leverage these algorithms. 
Instead, we develop new statistical model checking algorithms for temporal logics (over discrete and continuous time) that reason about sequences of distributions.

The scalability of our approach depends critically on the way the partition-based Mori-Zwanzig model reduction is performed, as it involves numerical integrations on the partitions. 
For stochastic hybrid systems with nonlinear but polynomial dynamics for the continuous part, the curse of dimensionality for direct numerical integration can be avoided as explicit symbolic solutions for the numerical solution exists. 
This is demonstrated in \cref{sec:case} by a case study.
Also, Monte-Carlo integrations can be adopted for more general dynamics with considerations on extra statistical errors.
Finally, we note that using this approach, we were the first to successfully
verify~\cite{rwwdv17} a highly non-linear model including lookup tables
of a powertrain control system that was proposed as a challenging problem
for verification tools by Toyota engineers~\cite{jin2014benchmarks}.

This paper is based on three of our previous papers~\cite{wang2015statistical,Wang2015267,wang2016verifying}, where discrete-time stochastic hybrid systems are studied in~\cite{wang2015statistical}; continuous-time stochastic (non-hybrid) systems are studied in~\cite{Wang2015267}; and continuous-time stochastic hybrid systems are studied in~\cite{wang2016verifying} without numerical evaluations.
This work presents a unification for the statistical verification of continuous-time and discrete-time stochastic hybrid systems, and provides a case study to numerically demonstrate the scalability of our statistical verification algorithms.

The rest of the article is organized as follows.
In \cref{sec:problem}, we introduce the general setup of the problem, including the definition of continuous-time stochastic hybrid systems and the syntax and semantics of metric interval temporal logic.
In \cref{sec:modelreduction}, we use the Mori-Zwanzig method to reduce the hybrid system to a Markov chain and prove that the temporal logic formulas on the hybrid system can be verified by checking slightly stronger formulas on the Markov chain.
In \cref{sec:mc}, we develop a statistical model checking algorithm for actually carrying out the verification.
In \cref{sec:model_reduction_of_discrete_hybrid_systems}, we consider discrete-time stochastic hybrid systems and derive similar model reduction and model checking results in this setting.
The scalability of the proposed algorithms is demonstrated by a case study in \cref{sec:case}.
Finally, we conclude in \cref{sec:conclusion}.

\section{Problem Formulation} \label{sec:problem}

We denote the set of natural, rational, non-negative rational, real, positive real, and non-negative real numbers by $\mathbb N$, $\Rat$, $\nnRat$, $\mathbb R$, $\mathbb R_{>0}$ and $\nnReal$ respectively.
We denote the essential supremum by ${\rm ess \ sup}$.
For $n \in \mathbb N$, let $[n] = \{1,2,\ldots,n\}$.
For any set $\mathbb{S}$, let $\mathbb{S}^{\pmb \omega}$ be the set of infinite sequences in $\mathbb{S}$. For $s \in \mathbb{S}^{\pmb \omega}$, let $s_i$ be the $i^{\rm th}$ element in the sequence.
For a finite set $A$, we denote the cardinality by $\Size{A}$ and its power set by $2^A$.
The empty set is denoted by $\emptyset$.
For $X \subset \mathbb R^d$, we denote the boundary of $X$ by $\partial X$.
The symbols $\mathbb P$ and $\mathbb E$ are used for the probability and the expected value, respectively.

\subsection{Stochastic Hybrid System} \label{sub:ctshs}

In this work, we follow the formal definitions of continuous-time stochastic hybrid systems in~\cite{Teel20142435,teel2015stochastic,Teel2017,Subbaraman2017} as shown in~\cref{fig:shs}.
However, we focus on a Fokker-Planck formulation and interpretation of the model.

\subsubsection{Continuous-time Stochastic Hybrid System}
\label{ssub:ctshs}
We denote the continuous and discrete states by $x \in \mathbb R^d$ and $q \in \mathcal{Q}$ respectively, where $\mathcal{Q} = \{q_1, \ldots, q_m\}$ is a finite set.
We call the combination $(q, x)$ the state of the system, and the product set $\mathbb{X} \subseteq \mathcal{Q} \times \mathbb R^d$ the state space.
For each $q \in \mathcal{Q}$, the state of the system flows in $\mathbf{A}_q\subseteq\Real^d$ and jumps forcedly on hitting the boundary $\mathbf{A}_q$.
We assume that each $\mathbf{A}_q$ is open and bounded, and the boundaries $\partial \mathbf{A}_q$ are second-order continuously differentiable.
On the flow set, the state $x$ of the system evolves by a stochastic differential equation
\begin{equation} \label{eq:continuous dynamics}
  \d x = f(q, \mathbf{x}) \d t + g(q, \mathbf{x}) \d B_t,
\end{equation}
where $q$ and $\mathbf{x}$ are random processes describing the stochastic evolution of the discrete and continuous states, and $B_t$ is the standard $n$-dimensional Brownian motion.
The vector-valued function $f$ specifies the drift of the state, and the matrix-valued function $g$ describes the intensity of the diffusion~\cite{karatzas2012brownian,revuz2013continuous}.
In~\eqref{eq:continuous dynamics}, we assume that $f(q,\cdot)$ and $g(q,\cdot)$ are locally Lipschitz continuous.
Meanwhile, the system jumps spontaneously by a non-negative integrable rate function $r (q, x)$ inside $\mathbf{A}_q$.
The probability distribution of the target of both spontaneous and forced jumps (as they happened on different domains) is given by a non-negative integrable target distribution $h (q', x', q, x)$, satisfying
\begin{equation}
    \sum_{q \in \mathcal{Q}} \int_{\mathbf{A}_q} h(q',x', q, x) \d x' = 1.
\end{equation}

\begin{figure}[t]
\centering
\includegraphics[width=0.5\textwidth]{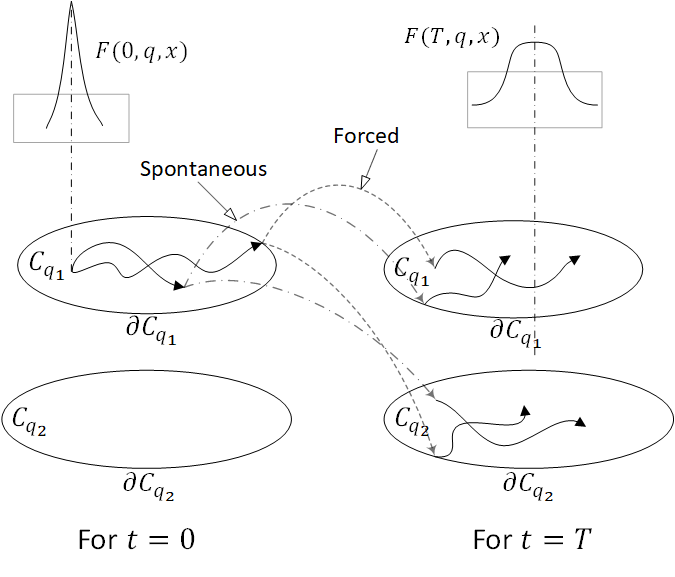}
\caption{A continuous-time stochastic hybrid system with two discrete states at time $0$ and $T$.} \label{fig:shs}
\end{figure}

\subsubsection{Fokker-Planck Equation} \label{ssub:fp_equations}
The probability distribution $F(t,q,x)$ of the state of the system in the flow set is determined by the Fokker-Planck equation, which can be derived in the same way as that for jump-diffusion processes~\cite{hanson_AppliedStochasticProcesses_2007},
\begin{equation} \label{eq:fp}
\begin{split}
  & \frac{\partial F(t,q,x)}{\partial t} =  L (F(t,q,x)) = \underbrace{- \sum_{a=1}^d \frac{\partial}{\partial x_a} (f_a(q,x) F(t,q,x))}_{\rm drift}
  \underbrace{+ \sum_{a=1}^d \sum_{b=1}^d \frac{\partial^2}{\partial x_a \partial x_b} \sum_{c=1}^d \frac{g_{ac}(q,x) g_{cb}(q,x) F(t,q,x)}{2}}_{\rm diffusion}
  \\ & \quad \underbrace{- r(q,x) F(t,q,x)}_\textrm{jump out}
  \underbrace{+ \sum_{q \in \mathcal{Q}} \int_{x \in \mathbb{A_q}} h(q, x, q', x') r(q',x') F(t, q', x') \d x'}_\textrm{spontaneous jump in},
  \underbrace{+ \sum_{q \in \mathcal{Q}} \int_{x \in \partial \mathbb{A_q}} h(q, x, q', x') (\mathbf{n} \cdot \mathbf{F}) \d x'}_\textrm{forced jump in},
\end{split}
\end{equation}
where $f_a$ is the $a$th element of $f$ from~\eqref{eq:continuous dynamics}, $\mathbf{n}$ is the unit vector pointing out of the flow set and the inner product $\mathbf{n} \cdot \mathbf{F}$ is the corresponding outgoing flow.
Here, $\mathbf{F}$ is a matrix (more precisely a second-order tensor) whose components are given by
\begin{equation}
  \mathbf{F}_{ab} = \frac{\partial}{\partial x_b} \sum_{c=1}^d \frac{g_{ac}(q,x) g_{cb}(q,x) F(t,q,x)}{2},
\end{equation}
where $a, b\in [d]$.
In \eqref{eq:fp}, $L$ is the Fokker-Planck operator for the system, and we write symbolically that $F(t,q,x) = e^{tL} F(0,q,x)$.
On the boundary, we have 
\begin{equation}
  F(t, q', x') = 0
\end{equation}
as it is absorbing (paths jump away immediately after hitting the boundary).
In the rest of the paper, we assume that the stochastic hybrid system given in this section is well defined in the sense that it gives a Fokker-Planck equation with a unique solution~\cite{karatzas2012brownian,revuz2013continuous}.

\subsubsection{Invariant Distribution} \label{ssub:invariant distribution}
An invariant distribution of the continuous-time stochastic hybrid system $F_{\rm inv}(q,x)$ is defined by
\begin{equation}
  L(F_{\rm inv}(q,x)) = 0.
\end{equation}
In this work, when handling temporal logic specifications of an infinite time horizon, we assume that $F(t,q,x)$ converges to the invariant distribution function $F_{\rm inv}(q,x)$ to ensure that the truth value of the specifications will not change after a finite time.

\subsubsection{System Observables} \label{ssub:system observables}
The state of the system is only partially observable. Here, we are interested in observables of the system given by
\begin{equation} \label{eq:obs}
  y(t) = \mathbb E [y(q(t),x(t))]
    = \sum_{q \in \mathcal{Q}} \int_{\mathbf{A}_q} \gamma(q,x) F(t,q,x) \d x,
\end{equation}
where $\gamma(q,x)$ is a weight function on $\mathbb X$, which is integrable in $x$ for each $q \in \mathcal Q$.

\begin{example} \label{ex:dynamics}
Throughout the paper, we use the following example to illustrate the theorems.
Consider a continuous-time stochastic hybrid system with two discrete states on $\mathbb{X} = \{1\} \times [0,1] \cup \{2\} \times [2,4]$.
It jumps uniformly to $[2,4]$ when hitting $x=0$ or $x=1$, and jumps uniformly to $[0, 1]$ when hitting $x=2$ or $x=4$.
It can jump spontaneously at any $x \in \mathbb{X}$ with the rate $h(x) = \mathbf{I}_\mathbb{X} (x) / 3$, where $\mathbf{I}_\mathbb{X} (\cdot)$ is the indicator function of the set $\mathbb{X}$.
In each location, the state of the system is governed by the stochastic differential equation
\begin{equation*}
  \d \mathbf{x} = \d t + \d B_t,
\end{equation*}
The probability distribution $F(t,q,x)$ of the state evolves by the Fokker-Planck equation
\begin{equation*}
\begin{split}
  & \frac{\partial F(t,q,x)}{\partial t} = - \frac{\partial F(t,q,x)}{\partial x} + \frac{1}{2}
  \frac{\partial^2 F(t,q,x)}{\partial x^2}
   + \frac{\partial F(t,q,0)}{\partial x} 
   - \frac{\partial F(t,q,1)}{\partial x} + \frac{1}{2} \frac{\partial F(t,q,2)}{\partial x} - \frac{1}{2} \frac{\partial F(t,q,4)}{\partial x}
\end{split}
\end{equation*}
with the boundary conditions
\begin{equation*}
  F(t,q,0) = F(t,q,1) = F(t,q,2) = F(t,q,4) = 0.
\end{equation*}
Initially, the state of the system is uniformly distributed on $[0, 1/2]$.
\end{example}

\subsection{Metric Interval Temporal Logic} \label{sub:mitl}

We are interested in verifying temporal properties 
of the continuous-time stochastic hybrid systems.
These properties are specified as follows.
The atomic propositions $\AP$ are inequalities 
of the form $y\sim c$ ($c\in\Rat$, $\sim\in\Brace{<,\leq,\geq,>}$),
where $y$ is an observable of the system given by \eqref{eq:obs};
and these atomic propositions are concatenated 
by the syntax of Metric Interval Temporal Logic (MITL)~\cite{96-MITL}.
This type of logic is also referred to as 
Signal Temporal Logic (STL)~\cite{04-STL,15-monSTL,13-monSTL} in the literature.
The syntax of MITL is given in~\cref{def:MITL-syn}.

\begin{definition}[MITL Syntax] \label{def:MITL-syn}
An MITL formula is defined using the following BNF form:
\[
\phi \Coloneqq
\bot                         \mid
\top                         \mid
y \sim c                          \mid
\phi\land\phi                \mid
\phi\lor \phi                \mid
\Until{\phi}{\phi}[I]   \mid
\Release{\phi}{\phi}[I],
\]
where $c\in\Rat$, $\sim\in\Brace{<,\leq,\geq,>}$ 
and $I$ is a non-singleton interval on $\nnReal$.
\end{definition}

We note that the syntax does not contain negation ($\neg$), since $\Brace{<,\leq,\geq,>}$ is closed under negation.
For a standard MITL formula, negation on non-atomic formulas can always be pushed inside as part of the atomic propositions. For example,
    $\neg\Paren{y > 0}$ is defined as $y \leq 0$,
    $\neg\Paren{\phi_1\vee\phi_2}$  is defined as $\Paren{\neg\phi_1} \land \Paren{\neg\phi_2}$, and
    $\neg\Paren{\Until{\phi}{\psi}[I]}$ is defined as $\Release{\Paren{\neg\phi}}{\Paren{\neg\psi}}[I]$.

The continuous-time stochastic hybrid system induces a signal $f(t): \nnReal \rightarrow 2^\AP$ by $(y \sim c) \in f(t)$ iff $y \sim c$ holds at time $t$.
The semantics of MITL are defined with respect to the signal $f(t)$ as follows.
  
\begin{definition}[MITL Semantics]\label{def:MITL-sem}
Let $\phi$ be an MITL formula and $f$ be a signal $f:\nnReal \rightarrow 2^\AP$. The satisfaction relation $\Sat$ between $f$ and $\phi$ is
defined according to the following inductive rules:
\[
\begin{array}{@{}l@{\ }c@{\ }l@{}}
f \models \bot & {\rm iff} & \mbox{false}
\\ f \models \top & {\rm iff} & \mbox{true}
\\ f \models y \sim c & {\rm iff} & \Paren{y \sim c} \in f(0)
\\ f \models \phi \land \psi & {\rm iff} & \Paren{f \models \phi} \land \Paren{f \models \psi}
\\ f \models \phi \lor \psi & {\rm iff} & \Paren{f \models \phi} \lor \Paren{f \models \psi}
\\ f \models \Until{\phi}{\psi}[I] & {\rm iff} & \exists t \in I, \Paren{f^t \models \psi}
\land \forall t' \in (0,t), f^{t'} \models \phi
\\ f \models \Release{\phi}{\psi}[I] & {\rm iff} & \forall t \in I, \Paren{f^t \models \psi} 
\textrm{ or } \exists t\in\mathbb R_{>0}, (f^t \models \phi \land \forall t'\in[0, t] \cap I, f^{t'} \models  \psi) \textrm{ or }
\\ & &
\exists t \in I', t' \in I \cap (t,\infty),
\forall t''\in I, (  t'' \leq t \to f^{t''} \models \psi)\land
(t<t'' \leq t'\to f^{t''} \models \phi)
\end{array}
\]
where $f^r (\cdot) = f(r + \cdot)$ and 
$I' = I \cup \{\underline{I}\}$ in the semantics of $\Release{\phi}{\psi}[I]$ with $\underline{I}$ being the lower bound of $I$. We define $\llbracket  \phi  \rrbracket $ to be the set of signals that satisfy $\phi$.
\end{definition}

Our semantics of $\Release{}{}$ in MITL is more complicated than 
the semantics of $\Release{}{}$ in LTL, its discrete-time counterpart.
Our semantics of $\Release{}{}$ is also different from 
the common semantics of MITL~\cite{96-MITL}.
This is because it has recently been shown that 
the common semantics of MITL cannot ensure that 
the formulas $\neg(\Until{\phi}{\psi}[I])$ 
and $\Release{(\neg\phi)}{(\neg\psi)}[I]$ are equivalent 
for the continuous-time domain
(see~\cite{roohi2018revisiting} for details).
Following the semantics of MITL, the satisfiability/model checking problems 
for MITL with \textit{abstract} atomic propositions 
are known to be \EXPSPACEcomp~\cite{96-MITL,roohi2018revisiting}.
The corresponding decision procedure has a close connection with timed automata.

\begin{definition}[\TAcp~\cite{94-TA}]\label{def:TA}
\TAus\ $A$ is a tuple
  $(\hQ, \hX, \hA, \hL, \hI, \hE, \hQi, \hQf)$ where
\begin{compactitem}
\item $\hQ$ is a finite non-empty set of {\em locations}.
\item $\hX$ is a finite set of {\em clocks}.
\item $\hA$ is a finite {\em alphabet}.
\item $\hL: \hQ \rightarrow \hA$ maps each location to the {\em label} of that location.
\item $\hI: \hQ \rightarrow (\hX \rightarrow {\mathbb I}_{\geq 0})$ maps each location to its {\em invariant} which is the set of possible values of variables in that location, {where ${\mathbb I}_{\geq 0}$ is the set of intervals on $\nnReal$}.
\item $\hE \subseteq \hQ\times\hQ\times 2^\hX$ is a finite set of {\em edges} of the form {$e = (s,d,j)$}, where $s = \hS e$ is {\em source} of the edge; $d = \hD e$ is {\em destination} of the edge; and $j = \hJ e$ is the set of clocks that are {\em reset} by the edge.
\item $\hQi \subseteq \hQ$ is the set of {\em initial locations}.
\item $\hQf \subseteq \hQ$ is the set of {\em final locations}.
\end{compactitem}
\end{definition}

A {\em run} of the \TAls\ $A$ is a sequence of tuples $\Paren{\rho, \tau, \eta} \in \hQ^\omega \times {\mathbb I}^\omega_{\geq 0} \times \hE^\omega$ in which the following conditions holds:
\begin{inparaenum}[(i)]
\item $\rho_0\in\hQi$, \ie, $\rho$ starts from an initial location $\hQi$;
\item $(\hS\eta_n=\rho_n)  \land (\hD\eta_n=\rho_{n+1})$, \ie, the source and destination of edge $\eta_n$ is $\rho_n$ and $\rho_{n+1}$, respectively;
\item $\tau_0, \tau_1, \ldots$ is an ordered and disjoint partition of the time horizon $\nnReal$; and
\item $\forall t \in \tau_n, x \in \hX$, we have $\varrho_n(x)+t- \underline \tau_n \in\hI(\rho_n,x)$, where $\varrho_{0}(x)=0$ and $\varrho_{n+1}(x)$ is inductively defined by
\begin{gather*}
\varrho_{n+1}(x)=
\begin{cases}
    0, &\text{ if }x\in\hJ\eta_n\\
    \varrho_n(x)+ \overline \tau_n- \underline \tau_n , &\text{ otherwise}
\end{cases}
\end{gather*}
\ie, clocks must satisfy the invariant of the current location.
\end{inparaenum}
Here, $\underline \tau$ and $\overline \tau$ are the lower and upper bound of the interval.

A run satisfying the condition $\Inf[\rho]\cap\hQf\neq\emptyset$, \ie, some location from $\hQf$ has been visited infinitely many times by $\rho$, is called an {\em accepting run} of $A$.
Note that every run of $A$ induces a function $f$ of type $\nnReal \rightarrow \hA$ that maps $t$ to $\hL(\rho_n)$, where $n$ is uniquely determined by the condition $t\in\tau_n$.
We define the {\em language} of $A$, denoted by $\Lang[A]$, to be the set of all functions that are induced by accepting runs of $A$.
The language of timed automata is closely related to MITL as follows.

\begin{lemma}[MITL to Timed Automata~\cite{96-MITL}]\label{thm:mitl2ta}
For any \MITL\ formula $\phi$, a \TAls\ $A_\phi$ can be constructed such that $\Lang[A_\phi] = \llbracket  \phi  \rrbracket $, \ie, the set of functions
that satisfy $\phi$ is exactly those that are induced by accepting runs of $A_\phi$.
\end{lemma}

\begin{example} \label{ex:specification}
Following Example~\ref{ex:dynamics}, we want to check the following MITL formula
\begin{align*}
    \phi_1 = \top\ \mc U \left( y_2(t) > \frac{1}{4} \right),
    \quad \phi_2 = \left( y_1(t) > \frac{1}{2} \right) \mc U \left( y_2(t) > \frac{1}{4} \right),
\end{align*}
where
\begin{align*}
    y_1(t) = \sum_{q \in \mathcal{Q}} \int_{\mathbf{A}_q} I_{[0,1]} F(t,q,x) \mathrm d x,
    \quad y_2(t) = \sum_{q \in \mathcal{Q}} \int_{\mathbf{A}_q} I_{[2,4]} F(t,q,x) \mathrm d x.
\end{align*}
\end{example}

\section{Model Reduction of Continuous-time Hybrid Systems} 
\label{sec:modelreduction}

The model reduction procedure for continuous-time stochastic hybrid system 
follows the three steps:
(i) reduce the dynamics by partitioning the state space; 
(ii) reduce the temporal logic specifications accordingly; and 
(iii) estimate the model reduction error.

\subsection{Reducing the Dynamics} \label{subsec:reduce_dynamics}

To implement the Mori-Zwanzig model reduction method 
\cite{chorin_optimal_2000} 
for continuous-time stochastic hybrid systems, 
we partition the continuous state space into finitely many partitions 
$\mathbb{S} = \{s_1,\ldots,s_n\}$, and treat each of them as a discrete state.
The idea of partitioning is similar to 
\cite{abate_ApproximateModelChecking_2010a,abate_ApproximateAbstractionsStochastic_2011}
for the discrete-time stochastic hybrid systems.
The partition is called an equipartition if they are hypercubes with the same size $\eta$.
We assume that for each $s_i$, there exists $q \in \mathcal{Q}$ such that $s_i \subset \{q\} \times \mathbf{A}_q$, and denote its measure by $\mu(s_i)$.
Let $m(\mathbb{X})$ and $m(\mathbb{S})$ be sets of probability distribution functions on $\mathbb{X}$ and $\mathbb{S}$, respectively.
Then we can define a projection $P: m(\mathbb{X}) \rightarrow m(\mathbb{S})$ and an injection $R: m(\mathbb{S}) \rightarrow m(\mathbb{X})$ between $m(\mathbb{X})$ and $m(\mathbb{S})$ by
\begin{equation} \label{eq:projection}
  p_j = (P F(q,x))_j = \int_{s_j} F(q,x) \d x,
\end{equation}
where $p_j$ is the $j$th element of $p$, and
\begin{equation} \label{eq:injection1}
  R p = \sum_{j=1}^{n} p_j \uni_{s_j},
\end{equation}
where $\uni_{s_j}$ is the uniform distribution on $s_j$:
\begin{equation} \label{eq:injection2}
  \uni_{s_j}(x) =
  \begin{cases}
    \frac{1}{\mu(s_j)}, &\text{ if } x \in s_j \\
    0, &\text{ otherwise.}
  \end{cases}
\end{equation}

Here the projection $P$ and the injection $R$ are defined for probability distributions.
But they extend naturally to $L_1$ functions on $\mathbb{X}$ and $\mathbb{S}$ respectively.
The projection $P$ is the left inverse of the injection $R$ but not \emph{vice versa}, namely $PR = I$ but $RP \neq I$.


This projection $P$ and injection $R$ can reduce the Fokker-Planck operator to a transition rate matrix on $\mathbb{S}$, and hence reduce the continuous-time stochastic hybrid system into a continuous-time Markov chain.
Following~\cite{chorin_optimal_2000}, the Fokker-Planck operator 
given in~\eqref{eq:fp} reduces to the transition rate matrix $A$ by
\begin{equation} \label{eq:reduce dynamics}
  A = P L R
\end{equation}
In practice, we are usually interested in 
a continuous state space $\mathbb{X}$ 
that is partitioned into hypercubes of edge length
$\eta$.
In this case, the transition rate matrix $A$ is
explicitly expressed as follows.


\begin{theorem}
Let $\mathbb{S} = \{s_1,s_2,\ldots,s_n\}$ be a partition%
\footnote{The partitions can be labeled by $S$ arbitrarily.}
of the $d$-dimensional continuous state space $\mathbb{X}$ 
into hypercubes of edge length $\eta$, 
and $P$ and $R$ be the corresponding projection and injection 
given by~\eqref{eq:projection}-\eqref{eq:injection2},
the transition rate from the state $s_i$ to the state $s_j$ 
($i \neq j$) at time $t$ is given by
\begin{equation} \label{eq:reduce_rate}
\begin{split}
  A_{ij} = \mathbf{n} \cdot \big( \mathbf{N} + \frac{\mathbf{M} \cdot \mathbf{n} (p_i - p_j)}{\eta} \big) + \mathbf{R}
\end{split}
\end{equation}
for $a,b \in [n]$, where $\mathbf{n}$ is (if exists) the unit vector of the boundary $s_i \cap s_j$ pointing from $s_i$ to $s_j$, 
$\mathbf{N}$ is a $d$ dimensional vector with components
\begin{equation} \label{eq:nn}
  \mathbf{N}_{a} = \int_{\partial s_i \cap \partial s_j} f_{a} (q,x) \d x,
\end{equation}
$\mathbf{M}$ is a $d \times d$ matrix with components
\begin{equation} \label{eq:mm}
  \mathbf{M}_{ab} = \int_{\partial s_i \cap \partial s_j} \sum_{c=1}^d \frac{g_{ac}(q,x) g_{cb}(q,x)}{2} \d x,
\end{equation}
and for an inner cell $s_i$,
\begin{equation} \label{eq:oo1}
  \mathbf{O} = \int_{s_i \times s_j} \frac{\mbf{I}_{s_j} (q,x) h(q, x, q', x') r(q', x') \mbf{I}_{s_i} (q',x')  }{\eta^d} \d x' \d x
\end{equation}
for a boundary cell $s_j$,
\begin{equation} \label{eq:oo2}
  \mathbf{O} = \int_{s_i \times s_j} \frac{\mbf{I}_{s_j} (q,x) h(q, x, q', x') \mathbf{n'} \cdot \mathbf{M} \cdot \mathbf{n'} p_i  }{\eta / 2;} \d x' \d x
\end{equation}
with $\mbf{I}_{s_i}$ being the indicator function of $s_i$ and $\mathbf{n'}$ being the vector pointing out of the boundary of the flow set.
\end{theorem}

\begin{proof}
For simplicity, we first show the proof for the 1D case.
Specifically, for fixed $q$, we integrate both sides of~\eqref{eq:fp} on the cell $I = [p, p + \Delta p]$, and apply the Stokes theorem for the first two terms, we derive
\begin{equation} \label{eq:ifp}
  \begin{split}
    & \int_{I} \frac{\partial F(t,q,x)}{\partial t} \d x = - f (q,x) F(t,q,x) \Big\vert^{p + \Delta_p}_p
    + \frac{\partial}{\partial x} \frac{g^2 (q,x) F(t,q,x)}{2} \Big\vert^{p + \Delta_p}_p
    - \int_{I} r(q,x) F(t,q,x) \d x
    \\ & \quad + \sum_{q \in \mathcal{Q}} \int_{x \in \mathbb{A_q}} \int_{I} h(q, x, q', x') r(q',x') F(t, q', x') \d x \d x'
    + \sum_{q \in \mathcal{Q}} \int_{x \in \partial \mathbb{A_q}} \int_{I} h(q, x, q', x') (\mathbf{n} \cdot \mathbf{F}) \d x \d x'.
  \end{split}
\end{equation}
The left-hand side of~\eqref{eq:ifp} is the rate of probability change in the cell $I$.
On the right-hand side of~\eqref{eq:ifp}, (i) the combination of the first two terms $f (q,x) F(t,q,x) - \frac{\partial}{\partial x} \frac{g^2 (q,x) F(t,q,x)}{2}$ is the probability flow on the boundary;
(ii) the other terms correspond to average probability jumps inside the cell $I$.
The same is true for multidimensional cases.

By applying~\eqref{eq:reduce dynamics}, it is easy to check the probability flow between adjacent cells sharing a boundary is~\eqref{eq:mm} and~\eqref{eq:nn}.
The probability of jumping from one inner cell to another cell 
has the rate~\eqref{eq:oo1}.
Finally, the probability of jumping from one boundary cell to another cell 
has the rate~\eqref{eq:oo2}.
Thus,~\eqref{eq:reduce_rate} holds.
\end{proof}

Roughly speaking, the transition rate between two partitions in the same location is the flux of $f(q,x)$ across the boundary and the transition rate between two different locations is the flux of $r(q,x)$.


\subsection{Reducing MITL Formulas} \label{subsec:reduce_mitl}

The observables on the continuous-time stochastic hybrid system reduce to the corresponding continuous-time Markov chain using the projection $P$.
Let $y$ be an observable on the continuous-time stochastic hybrid system with weight function $\gamma(q,x)$.
To facilitate further discussion, we assume that $\gamma(q,x)$ is invariant under the projection $P$, \ie,
\begin{equation} \label{eq:ass}
  \gamma(q,x) = R P \gamma(q,x),
\end{equation}
which means that the function $\gamma(q,x)$ 
can be written as the linear combination 
of the indicator functions of the partitions
(sometimes called a \emph{simple function}.)
We define a corresponding observable $y'$ on the continuous-time Markov chain that derives from the model reduction procedure by
\begin{equation} \label{eq:rap}
\begin{split}
    y'(0) & = \sum_{q \in \mathcal{Q}} \int_{\mathbf{A}_q} \gamma(q,x) P F(0,q,x) \d x
     = \sum_{i=1}^n \paren{\int_{s_i} \gamma(q,x) \d x} \paren{\int_{s_i} F(0,q,x) \d x}
     = \sum_{i=1}^n r_i p(i).
\end{split}
\end{equation}

From now on, we will always denote the corresponding observable on the CTMC by $y'$ for any observable $y$ on the continuous-time stochastic hybrid system.

\subsection{Reduction Error Estimation} \label{subsec:error}

For a given observable $y$ with weight function $\gamma(q,x)$, the error of the projection $P$ with respect to the observable $y$ is defined by the maximal possible difference between $y$ and $y'$,
\begin{equation} \label{eq:projection error}
\begin{split}
  \Delta_y = \Big\vert \sum_{q \in \mathcal{Q}} \int_{\mathbf{A}_q} \gamma(q,x) (F(0,q,x) - RP F(0,q,x))  \d x \Big\vert.
\end{split}
\end{equation}

\begin{remark}
When refining the partition of $\mathbb{X}$, $RP \rightarrow I$ in the weak operator topology~\cite{rudin1973functional};
that is, any distribution function $F(q,x)$ on the state space, $\abs{ \sum_{q \in \mathcal{Q}} \int_{\mathbf{A}_q} \gamma(q,x) (F(q,x) - RPF(q,x)) } \rightarrow 0$ holds for any measurable weight function $\gamma(q, x)$.
Accordingly for~\eqref{eq:projection error}, $\Delta_y \rightarrow 0$ for any given $y$.
\end{remark}

By the definition of $\Delta_y$, we know that, at the initial time, the atomic propositions on the continuous-time stochastic hybrid system and the CTMC have the relations
\begin{align*}
& y(0) > c \Longrightarrow  y'(0) > c - \Delta_y, \quad y(0) < c \Longrightarrow  y'(0) < c + \Delta_y,
\end{align*}
and similarly,
\begin{align*}
& y'(0) > c + \Delta_y \Longrightarrow  y(0) > c, \quad y'(0) < c - \Delta_y \Longrightarrow  y(0) < c.
\end{align*}

To derive the relations of the observables between the continuous-time stochastic hybrid system and the CTMC at any time,
we define the reduction error of the observable $y$ at time $t$ due to the model reduction process by
\begin{equation} \label{eq:reduction error}
\begin{split}
  \Theta_y(t) = \abs{y(t) - y'(t)} = \Big\vert \sum_{q \in \mathcal{Q}} \int_{\mathbf{A}_q} \gamma(q,x) (e^{Lt} - R e^{At} P) F(0,q,x) \d x \Big\vert,
\end{split}
\end{equation}
where $F(0,q,x)$ is an initial distribution of the continuous-time stochastic hybrid system and $y'(t)$ is the corresponding observable of $y(t)$ on the CTMC.
This reduction error is illustrated in~\cref{fig:reduction error}.
Note that the diagram is not commutative; the difference between going along the two paths is related to the reduction error.

\begin{figure}[t]
\centering
\begin{tikzpicture}
  \matrix (m) [matrix of math nodes,row sep=3em,column sep=4em,minimum width=2em] {
     F(0,q,x) & F(t,q,x) \\
     p(0) & p(t) \\};
  \path[-stealth]
    (m-2-1) edge[dashed] node [left] {$R$} (m-1-1)
    (m-1-1) edge node [below] {$e^{Lt}$}  (m-1-2)
    (m-2-1) edge node [below] {$e^{At}$} (m-2-2)
    (m-1-2) edge[dashed] node [right] {$P$} (m-2-2);
\end{tikzpicture}
\caption{Diagram for reduction error.}  \label{fig:reduction error}
\end{figure}
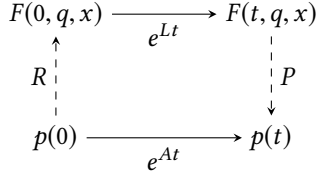

In general, the reduction error $\Theta(t)$ may not be bounded as $t \rightarrow \infty$.
To find a sufficient condition for boundedness, we define the reduction error of the Fokker-Planck operator $L$ by
\begin{equation} \label{eq:delta}
  \delta(t,q,x) = (L - R P L)  e^{t R P L} F(0,q,x).
\end{equation}
Accordingly, we define the integration of $\delta(t,q,x)$ with respect to the weight function $\gamma(q,x)$ by
\begin{equation} \label{eq:rate error}
\begin{split}
\Lambda_y =  \sup_{t\geq 0} \Big\vert \sum_{q \in \mathcal{Q}} \int_{\mathbf{A}_q} \gamma(q,x) (L - R P L)  e^{t R P L} F(0,q,x) \d x \Big\vert,
\end{split}
\end{equation}
which captures the maximal change of the time derivative of observable $y$.
When the reduction error $\delta(t,q,x)$ converges exponentially in time, an upper bound of the reduction error $\Theta(t)$ can be obtained.

\begin{definition} \label{def:cont}
For $\alpha > 0$, $\beta \geq 1$ and a given observable $y$, the continuous-time stochastic hybrid system is $\alpha$-contractive with respect to $y$, if for any initial distribution function $F(0,q,x)$ on the state space, we have
\begin{equation} \label{eq:cont}
\begin{split}
& \Big\vert \sum_{q \in \mathcal{Q}} \int_{\mathbf{A}_q} \gamma(q,x) e^{tL} \delta (t,q,x) \d x \Big\vert
\leq \beta e^{- \alpha t} \Big\vert \sum_{q \in \mathcal{Q}} \int_{\mathbf{A}_q} \gamma(q,x) \delta (t,q,x) \d x \Big\vert.
\end{split}
\end{equation}
where $\delta (t,q,x)$ is given by~\eqref{eq:delta}.
\end{definition}

This contractivity condition is to ensure that the model reduction error is bounded for all time, which is required for approximately keeping the truth value of temporal logic specifications of an infinite time horizon.
Although the condition seems restrictive, it is valid for a relatively wide range of systems including asymptotically stable systems. It is a commonly-used sufficient condition to guarantee the existence and uniqueness of an invariant measure for general dynamical systems, and the contractivity factor $\alpha$ is usually derived case-by-case. Using~\cref{def:cont}, we obtain the following theorem.

\begin{theorem} \label{thm:reduction error}
If the continuous-time stochastic hybrid system from \cref{ssub:ctshs} is $\alpha$-contractive, then for any $t \geq 0$, the reduction error $\Theta_y(t)$ for an observable $y$ satisfies
\begin{equation} \label{eq:uniform eb}
  \Theta_y(t) \leq \frac{\beta \Lambda_y}{\alpha} + \Delta_y.
\end{equation}
\end{theorem}

\begin{proof}
By Dyson's formula~\cite{chorin_optimal_2000}, we can decompose the exponential of $L$ by
\begin{equation} \label{eq:dyson}
  e^{tL} =  e^{t R P L} + \int_{[0,t]} e^{(t-\tau )L} (L - R P L) e^{\tau R P L} \d \tau,
\end{equation}
which can be verified by taking time derivatives on both sides.
Substituting~\eqref{eq:dyson} into~\eqref{eq:reduction error} gives
\begin{equation} \label{eq:ee}
\begin{split}
  \Theta_y(t) \leq & \Big\vert \sum_{q \in \mathcal{Q}} \int_{\mathbf{A}_q} \gamma(q,x) (e^{t R P L} - R e^{t A} P) F(0,q,x) \d x \Big\vert \\
  & + \Big\vert \sum_{q \in \mathcal{Q}} \int_{\real^d \times [0,t]} \gamma(q,x) e^{(t-\tau)L} (L - R P L) e^{\tau R P L} F(0,q,x) \d \tau \d x \Big\vert
\end{split}
\end{equation}
Since the projection $P$ and the injection $R$ preserve the $L_1$ norm, $RPL$ is also a Fokker-Planck operator. Noting $R e^{t A} P F(0,q,x) = e^{t R P L} P F(0,q,x)$, by~\eqref{eq:projection error}, we see that the first term on the right-hand side of~\eqref{eq:ee} is less than $\Delta_y$.

For the second term on the right-hand side of~\eqref{eq:ee}, 
by~\eqref{eq:rate error}-\eqref{eq:cont},
we have
\begin{equation}
\begin{split}
   & \Theta_y(t) \leq \Delta_y + \Big\vert \sum_{q \in \mathcal{Q}} \int_{\mathbf{A}_q}  \int_{[0,t]} \gamma(q,x) e^{(t-\tau)L} \delta(\tau,q,x) \d \tau \d x \Big\vert \\
  & \leq \Delta_y + \Big\vert \sum_{q \in \mathcal{Q}} \int_{\mathbf{A}_q}  \int_{[0,t]} \beta e^{- \alpha (t-\tau)} \gamma(q,x) \delta(\tau,q,x) \d \tau \d x \Big\vert \leq \frac{\beta \Lambda_y}{\alpha} + \Delta_y.
\end{split}
\end{equation}
\end{proof}

\cref{thm:reduction error} implies the following relations between the atomic propositions on the continuous-time stochastic hybrid system and the CTMC.

\begin{theorem} \label{thm:ctshs_main}
If the continuous-time stochastic hybrid system given in \cref{ssub:ctshs} is $\alpha$-contractive, then we have
\begin{align}
& y(t) > c \Longrightarrow  y'(t) > c - \Big( \frac{\beta \Lambda_y}{\alpha} + \Delta_y \Big), \\
& y(t) < c \Longrightarrow  y'(t) < c + \Big( \frac{\beta \Lambda_y}{\alpha} + \Delta_y \Big),
\end{align}
and similarly,
\begin{align}
& y'(t) > c + \Big( \frac{\beta \Lambda_y}{\alpha} + \Delta_y \Big) \Longrightarrow  y(t) > c, \label{eq:rule1} \\
& y'(t) < c - \Big( \frac{\beta \Lambda_y}{\alpha} + \Delta_y \Big) \Longrightarrow  y(t) < c. \label{eq:rule2}
\end{align}
\end{theorem}

In \cref{thm:ctshs_main}, the term $\Delta_y$ bounds the initial model reduction error and the term $\frac{\beta \Lambda_y}{\alpha}$ bounds the model reduction error accumulated over time.
Following \cref{thm:ctshs_main}, to verify an MITL formula $\phi$ for an $\alpha$-contractive continuous-time stochastic hybrid system introduced in \cref{ssub:ctshs}, we can strengthen $\phi$ to $\psi$ by replacing the atomic propositions according to~\eqref{eq:rule1}-\eqref{eq:rule2}.
If $\psi$ holds for the CTMC derived from the continuous-time stochastic hybrid system following the model reduction procedure of \cref{subsec:reduce_dynamics,subsec:reduce_mitl}, then $\phi$ holds for the continuous-time stochastic hybrid system.

\begin{example} \label{ex:model reduction}

Following Example~\ref{ex:dynamics} and \ref{ex:specification}, the invariant distribution of this process is $F_{\rm inv} = \uni_\mathbb{X} / 3$.
We partition $\mathbb{X}$ into intervals of length $1/N$. By the above model reduction procedure it reduces to a CTMC with transition {rate} matrix $M$ given by
\begin{equation*}
    M_{ij} = \frac{\delta_{ij}}{4} + \frac{1}{4 N}
\end{equation*}
where $i \in [3N]$ and $j \in [3N]$. The invariant distribution $F_{\rm inv}$ remains unchanged, and the MITL formula to check is
\begin{align*}
    & \phi_1' = \top\ \mc U \left( y_2'(t) > \frac{1}{4} + {\Theta_y(t)} \right)
    \\ & \phi_2' = \left( y_1'(t) > \frac{1}{2} + {\Theta_y(t)} \right) \mc U \left( y_2'(t) > \frac{1}{4} + {\Theta_y(t)} \right)
\end{align*}
where ${\Theta_y(t)}$ is the model reduction error and
\begin{align*}
    y_1'(t) = \sum_{i=1}^{N} p_i(t),
    \quad y_2'(t) = \sum_{i=2N+1}^{3 N} p_i(t).
\end{align*}
When $N=30$, we have ${\Theta_y(t)} \leq 0.02$ from~\eqref{eq:projection} and~\eqref{eq:rate error}.

\end{example}

\section{Statistical Model Checking of MITL} 
\label{sec:mc}

Let $C$ be the CTMC derived from the model reduction
(\cref{sec:modelreduction})
and $\varphi$ be corresponding reduced MITL formula.
In this section, we propose a statistical model checking algorithm
to verify the formula $\varphi$ for the CTMC $C$.

We denote the set of atomic propositions 
contained in $\varphi$ by $\AP_\varphi$.
The pair $C, \varphi$ can generate a \emph{signal} 
by evaluating the truth value of
the atomic propositions in $\AP_\varphi$
on the CTMC $C$ for each time. 
We use $\Sem{C, \AP_\varphi}$ to denote the singleton set 
that contains this signal.
Let $T_{C,\varphi}$ be the \TAls\ such that 
$\Sem{C,\AP_\varphi} \subseteq \Lang[T_{C,\varphi}]$.
%
Using \Cref{thm:mitl2ta}, we construct 
two \TAlp\ $T_{\varphi}$ and $T_{\neg\varphi}$
such that their languages are 
the signals accepted and rejected by $\varphi$, respectively.
%
%
If the intersection of $\Lang[T_{C,\varphi}]$ and $\Lang[T_\varphi]$ is empty then 
$C$ violates $\varphi$. Similarly,
if the intersection of $\Lang[T_{C,\varphi}]$ and $\Lang[T_{\neg\varphi}]$ is empty then 
$C$ satisfies $\varphi$.
This emptiness problem for the intersection of \TAlp\ 
is known to be \PSPACEcomp~\cite{94-TA}.
However, it is possible that none of the two intersections is empty.
To avoid this situation, we assume that each signal of $\Lang[T_{C,\varphi}]$ 
remains close to the signal in $\Sem{C, \AP_\varphi}$.
That is, if the signal in $\Sem{C, \AP_\varphi}$ satisfies/violates $\varphi$,
then there is a close signal that violates/satisfies $\varphi$.
We will formalize this later.

We use a statistical method to construct the \TAls\ $T_{C,\varphi}$.
Let $\mY(t)$ be the probability distribution of the state of the CTMC $C$, 
and $f(t)$ be the set of atomic propositions that $\mY(t)$ satisfies at the time $t \in [0,\infty)$.
%
%
Since the CTMC converges to a unique invariant distribution $\inv{\mY}$, there exists
a known constant $\delta'\oftype\Real$ and 
a known estimation $\mY^*$ of $\inv{\mY}$ 
such that
\begin{itemize}
\item $\forall(r\cdot p(t) \sim c) \in \AP_\varphi, \Size{\mR\cdot\inv{\mY}-\mC}>\delta'$, and
\item $\LOne{\inv{\mY}-\mY^*}<\frac{\delta'}{3}$,
where $\LOne{\cdot}$ is the $\ell_1$ norm.
\end{itemize}

For each atomic proposition $(y \sim c)$, 
where $y$ is of the form $r \cdot p(t)$, 
we assume \wLOG\ that 
\begin{itemize}
\item $\mR$ is not identical to $\bf0$ (otherwise, $(y\sim c)$ can be replaced with $\top$ or $\bot$); and
\item the maximum absolute value in $\mR$ is exactly $1$ (by scaling the parameters in $(y \sim c)$).
\end{itemize}
Furthermore, let $T$ be a time such that $\LOne{\mY(T)-\mY^*}<\frac{\delta'}{3}$ holds
(we will show how to find $T$ later in this section).
For any $t \geq T$, we have 
$\LOne{\mY(t)-\inv{\mY}}<\frac{2 \delta'}{3}$.
Also, we assume that $\mR\cdot\inv{\mY}-c > \delta'$ holds
(the discussion for $\mR\cdot\inv{\mY}-c<-\delta'$ is similar).
By $\abs{\mR\cdot\mY(t)-\mR\cdot\inv{\mY}}\leq\LOne{\mY(t)-\inv{\mY}}<\frac{2\delta'}{3}$,
we know $r\cdot\mY(t)-c>\frac{\delta'}{3}$.
Then by $\abs{\mR\cdot\mY(t)-\mR\cdot\mY^*}\leq\LOne{\mY(t)-\mY^*}<\frac{\delta'}{3}$,
we have
$r\cdot\mY^*>c$.
%
Therefore, the truth value of $(y\sim c)$ is fixed for any $t > T$ 
and can be determined by looking at $\mY^*$.
%

We use \Cref{alg:truncate} to find a time $T$ such that $\mY(T)$
is $\frac{\delta'}{3}$-close to $\mY^*$ (the estimation of the invariant distribution).
Our statistical algorithm compares $\mY(T)$ and $\mY^*$ for successively larger values of $T$ 
until $\LOne{\mY(T)-\mY^*}<\frac{\delta'}{3}$ holds.
To check if two distributions are close,
we employ \Cref{thm:dist}.
When $\LOne{\mY(t)-\mY^*}>\frac{\delta'}{3}$,
starting from the iteration $i=1$,
the probability of \Cref{thm:dist} not rejecting $t$ 
is at most $\alpha' \times2^{-i}$.
Thus, the probability of returning a wrong time $T$ is at most $\alpha'$. 

\begin{lemma}\cite{dist-closeness-2013-Batu}\label{thm:dist}
For any $\alpha,\delta>0$, and 
any two distributions $\mY$ and $\mY'$ on $n$ discrete values, 
there is a test $\ALGClose(\mY, \mY', \alpha, \delta)$ which runs in 
time $O\Paren*{n^{2/3} \delta^{-8/3}\log\Paren*{n/\alpha}}$ such that 
(i) if $\LOne{\mY-\mY'}\leq \max\Paren*{\frac{\delta^{4/3}}{32\sqrt[3]{n}},\frac{\alpha}{4\sqrt{n}}}$,
then the test accepts with probability at least $1-\alpha$; 
and (ii) if $\LOne{\mY-\mY'} > \delta$, 
then the test rejects with probability at least $1-\alpha$.
\end{lemma}

\begin{algorithm}[t]
\KwData{%
  CTMC $(C,\mY_0)$,
  estimation of invariant distribution $\mY^*$,
  Atomic formula $(y\sim c)$,
  parameters $\alpha'$, and $\delta'$}
\Function{$\mathtt{DurationOfSimulation}$}{%
  $t \gets 1$\\
  \While{$\ALGClose\Paren*{\mY(t), \mY^*, \frac{1}{2}\alpha', \frac{\delta'}{3}} = \mathtt{failed}$}
  {
    $t \gets 2\times t$\\
    $\alpha' \gets \frac{1}{2}\alpha'$
  }
  \Return{t+1}}
\caption{Truncating time horizon}\label{alg:truncate}
\end{algorithm}

Before constructing the \TAls\ for times within $[0,T]$, we first explain how to 
statistically verify if $\mY(t)$ satisfies an atomic proposition $(y\sim c)$.
%
For now, assume that elements of $\mR$ are from $\Brace{0,1}$.
Then, $\mY(t)$ satisfies $(y\sim c)$ \iFF\
the probability of drawing a state $s$ from $\mY(t)$ with $\mR(s)=1$ 
is great than $c$.
This can be statistically checked by drawing samples from $\mY(t)$ and 
using the sequential probability ratio test (SPRT) 
\cite{sprt-1945-wald,sen_statistical_2005,yesno-to-yesnounknown-2006-Younes}.
It requires as input 
  an indifference parameter $\delta\oftype(0,1)$, and 
  the error bounds $\alpha,\gamma\oftype(0,1)$. 
The output of this test, called $\ALG_0$, is \Yes, \No, or \Unknown\ with the following guarantees:
\begin{subequations}
\newcommand{\LongFML}{\ensuremath{\Size*{\mR\cdot\mY(t)-c} > \delta}}%
\begin{align}
\Prb[\mathtt{res}=\makebox[0pt][l]{\No}     \phantom{\Unknown} \mid\makebox[0pt][l]{\ensuremath{\mR\cdot\mY(t)    > c}}\phantom{\LongFML}]\leq\alpha, \\
\Prb[\mathtt{res}=\makebox[0pt][l]{\Yes}    \phantom{\Unknown} \mid\makebox[0pt][l]{\ensuremath{\mR\cdot\mY(t)\not> c}}\phantom{\LongFML}]\leq\alpha, \\
\Prb[\mathtt{res}=\makebox[0pt][l]{\Unknown}\phantom{\Unknown} \mid\makebox[0pt][l]{\ensuremath{\LongFML}}             \phantom{\LongFML}]\leq\gamma.
\end{align}
\end{subequations}
The parameters $\alpha,\gamma,\delta$ can be made arbitrarily small
at the cost of requiring more samples. 
For the general case that the elements of $\mR$ are real numbers, 
the SPRT is not applicable.
Instead, we can use a technique due to Chow and Robbins
\cite{chow-robbins-1965}.

Given that $T$ is known,
we construct the \TAls\ for the time interval $[0,T]$.
For simplicity, we focus on constructing $T_{C,\Brace{P}}$ for an atomic proposition 
$P: y = \sum_{i=1}^n \mR_i p_i > c$, denoted by the pair $\Paren{\mR, c}$.
Then, at every time $t$, $f(t)$ is either the emptyset or $\Brace{\Paren{\mR, c}}$.
Let $T_{C,\Brace{P}}(t)$ be the set of reachable locations of $T_{C,\{P\}}$ at time $t$.
Given the parameters $\delta>0$, 
let $\Delta>0$ be a value at most 
$\frac{\delta}{3}\max\Brace*{\Size*{\frac{\mathtt{d}}{dt}(\mR\cdot\mY)(t)}\mid t\in[0,T]}^{-1}$ 
($\Delta$ can be set to $\frac{\delta}{3}\LInf{r}\LOne{M}$, where 
$\LInf{\cdot}$ and $\LOne{\cdot}$ are respectively $\ell_\infty$ and $\ell_1$ induced norms).
For any $t \in [0,T]$ and $t' \in [t-\Delta,t+\Delta] \cap[0,T]$, we have
\begin{enumerate}
\item if $\mR\cdot\mY(t) - c >  \frac{\delta}{3}$  then $\mR\cdot\mY(t') > c$, 
\item if $\mR\cdot\mY(t) - c < -\frac{\delta}{3}$  then $\mR\cdot\mY(t') < c$, 
\item if $\Size{\mR\cdot\mY(t)-c}\leq\frac{2\delta}{3}$  then $\Size{\mR\cdot\mY(t')-c} \leq \delta$.
\end{enumerate}
We partition $[0,T)$ into $\Floor*{\frac{T}{2\Delta}}+1$ intervals,  each of size strictly less than $2\Delta$.
Let $[t_1,t_2)$ be one of these intervals and define $t = \frac{1}{2}(t_1+t_2)$.
We then run $\ALG_0$ twice as follows, where $\alpha'$ and $\gamma'$ are obtained by dividing 
input parameters $\alpha$ and $\gamma$ over $\Ceil*{\frac{T}{2\Delta}}$.
\begin{align*}
  \mathtt{res}_1=\ALG_0\Paren*{\mR\cdot\mY(t), c+\frac{\delta}{3},\frac{1}{\Size{\AP_\varphi}}\alpha', \frac{1}{\Size{\AP_\varphi}}\gamma', \frac{\delta}{3}}, \\
  \mathtt{res}_2=\ALG_0\Paren*{\mR\cdot\mY(t), c-\frac{\delta}{3},\frac{1}{\Size{\AP_\varphi}}\alpha', \frac{1}{\Size{\AP_\varphi}}\gamma', \frac{\delta}{3}},
\end{align*}
If $\mathtt{res}_1=\Yes$, 
then $\forall t' \in [t_1,t_2), \Paren{\mR\cdot\mY(t')>c}$ holds with a bounded error $\alpha'$, so we set $T_{C,\{P\}}(t)=\{P\}$.
If $\mathtt{res}_2=\No$, 
then $\forall t' \in [t_1,t_2),\Paren{\mR\cdot\mY(t')<c}$ holds with a bounded error $\alpha'$, so we set $T_{C,\{P\}}(t)=\{\emptyset\}$.
Otherwise, for any time $t'$ in the interval, 
$\Size{\mR\cdot\mY(t')-c}\leq\delta$ with a bounded error $\gamma'$,
so we set
\begin{itemize}
\item $T_{C,\{P\}}(t)=\{q,q'\}$,
\item $\hL(q)=\{P\}$ and $\hL(q')=\emptyset$,
\item entry to $q$ or $q'$, and
\item switches between $q$ and $q'$ when their common invariant permits.
\end{itemize}
This ensures that within $[t_1,t_2)$, both states $q,q'$ can be reached and they can switch arbitrary many times. 
Intuitively, this means the atomic propositions 
within this interval are unknown and not fixed.
The algorithm to construct $T_{C,\varphi}$
is given by \Cref{fig:alg-part2-ct}.

Based on \Cref{alg:truncate,fig:alg-part2-ct},
the complete algorithm $\ALG$ to statistically verify  
the MITL formula $\varphi$ for the CTMC $C$
with the parameters $\delta,\delta',\alpha,\gamma$,
is given by the explanation at the beginning of this section.
The parameters $\delta'$ and $\frac{1}{2}\min\Brace{\alpha,\gamma}$ are given to \Cref{alg:truncate}, and
the parameters $\delta$, $\frac{1}{2}\alpha$, $\frac{1}{2}\gamma$ are given to \Cref{fig:alg-part2-ct}.
We have the following guarantee 
on the return $\mathtt{res}$ 
of the complete algorithm $\ALG$:
\begin{align}
  & \Prb[\mathtt{res}=\No\ \,  \mid C \Sat \varphi] \leq \alpha     \label{fml:type-1}  \\
  & \Prb[\mathtt{res}=\Yes     \mid C \nSat\varphi] \leq \alpha     \label{fml:type-2}  
\end{align}

As for the \Unknown\ output, let $B^{\delta}(\mR\cdot\mY)$ be the 
tube of functions that are point-wise $\delta$-close to $\mR\cdot\mY$
(formally, a function $f:\nnReal\goesto\Real$ is in $B^{\delta}(\mR\cdot\mY)$ \iFF\ for any $t\oftype\nnReal$, $\Size{f(t)-r\cdot\mY(t)}\leq\delta$).
The algorithm guarantees that
\begin{align}
\big(\forall \sigma\oftype B^{\delta}(\mR\cdot\mY), \ \sigma\Sat\varphi\big) \implies
\Prb[\mathtt{res}{=}\Unknown] \leq \alpha+\gamma\label{fml:type-3}\\
\big(\forall \sigma\oftype B^{\delta}(\mR\cdot\mY), \ \sigma\nSat\varphi\big) \implies
\Prb[\mathtt{res}{=}\Unknown] \leq \alpha+\gamma\label{fml:type-4}
\end{align}
Intuitively, if all the functions that are close to $\mR\cdot\mY$ satisfy $\varphi$ or none of them does then the probability of returning \Unknown\ is at most $\alpha+\gamma$.%
\footnote{There is a slight abuse of notation in \eqref{fml:type-3} and \eqref{fml:type-4}.
They use a function of type $\nnReal\goesto\Real$. However, $\Sat$ requires a signal (function 
of type $\nnReal\goesto2^\AP$). The signal contains atomic proposition $y\sim c$ at time $t$ \iFF\
$y(t)\sim c$ holds.}
 
\begin{algorithm}[t]
$h \gets \max\Brace*{\Size*{\frac{\mathtt{d}}{dt}(\mR\cdot\mY)(t)}\mid t\in[0,T]}$\\  
$\Delta \gets \frac{\delta}{3h}$\\
$n \gets \Size*{\AP_\varphi}\Ceil*{\frac{T}{2\Delta}}$\\
$T_{C,\{P\}} \gets $ an empty automaton\\
$\hX \gets \{t\}$, $q_{\rm last} \gets \bot$

\ForAll{$i\gets 0$ to $\Floor*{\frac{T}{2\Delta}}$}{
$\mathtt{res_1} \gets 
\ALG_0\Paren*{\mR\cdot\mY\Paren*{(i+\frac{1}{2})2\Delta},c+\frac{\delta}{3},\frac{\alpha}{2n},\frac{\beta}{2n},\frac{\delta}{3}}$
\\
$\mathtt{res_2} \gets 
\ALG_0\Paren*{\mR\cdot\mY\Paren*{(i+\frac{1}{2})2\Delta},c-\frac{\delta}{3},\frac{\alpha}{2n},\frac{\beta}{2n},\frac{\delta}{3}}$
\\
add a new location $q$ to $\hQ$

\uIf{$\mathtt{res}_1=\Yes$}{$\hL(q)\gets\{P\}$}
\uElseIf{$\mathtt{res}_2=\No$}{$\hL(q)\gets \emptyset$}
\uElse{$\hL(q)\gets \Unknown$}

$\hI(q) \gets 2i\Delta\leq t < 2(i+1)\Delta$

\uIf{$q_{\rm last} \neq \bot$}{$\hE\gets\hE\cup\Brace*{(q_{\rm last},q,\emptyset)}$}
\uElse{$\hQi\gets \{q\}$}
 
$q_{\rm last} = q$
}
add a new location $q$ to $\hQ$\\
$\hI(q)\gets \mathtt{true}$, $\hQf\gets\Brace{q}$\\
$\hE\gets\hE\cup\Brace*{(q_{\rm last},q,\emptyset),(q,q,\emptyset)}$\\
\uIf{$\mR\cdot\inv{\mY}>c$}{$\hL(q) \gets \{P\}$}
\uElse{$\hL(q) \gets \emptyset$}

$T_{C,\{P\}} \gets$ replace any \Unknown\ location in $\hQ$ with $q$ and $q'$ labeled $\{P\}$ and $\emptyset$. Duplicate edges from/to $q$ and $q'$ accordingly

Add $(q,q',\emptyset)$ and $(q',q,\emptyset)$ to $\hE$ for every split locations in the previous step.

\Return{$T_{C,\{P\}}$}

\caption{Constructing the \TAls\ $T_{C,\varphi}$}
\label{fig:alg-part2-ct}
\end{algorithm}

\begin{example}
Following Example~\ref{ex:model reduction}, we run our algorithm on the CTMC and derive that both $\phi_1'$ and $\phi_2'$ are true. This implies that the formulas $\phi_1$ and $\phi_2$ given in Example~\ref{ex:specification} are true on the system given in Example~\ref{ex:dynamics}.
\end{example}

\section{Discrete Hybrid Systems}
\label{sec:model_reduction_of_discrete_hybrid_systems}

In this section, we study the verification of temporal properties for discrete-time stochastic hybrid systems.
We follow the formulation of the discrete-time stochastic hybrid systems from~\cite{abate_ApproximateModelChecking_2010a,abate_ApproximateAbstractionsStochastic_2011} and use the inequality linear temporal logic (iLTL)~\cite{04-iLTL} to capture the temporal properties of interest.
The iLTL specifications are verified on the discrete-time stochastic hybrid systems by model reduction and statistical model checking in a similar way as \cref{sec:modelreduction,sec:mc}.

\paragraph*{Discrete-time stochastic hybrid systems}
Following the formulation of~\cite{abate_ApproximateModelChecking_2010a}, we focus on a Fokker-Planck formulation and interpretation of the model.
Using the notations from \cref{ssub:ctshs},
the dynamics of the system is captured by 
the initial distribution $F(0,q,x)$ 
on the state space $\mathbb{X} \subseteq \mathcal{Q} \times \mathbb R^d$
and the transition function $T(q',x',q,x)$,
which satisfies
\begin{equation}
  \sum_{q \in \mathcal{Q}} \int_{\mathbf{A}_q} T(q',x',q,x) \d x' = 1,
\end{equation}
for any $(q,x) \in \mathbb{X}$.
The transition function $T(q',x',q,x)$ can be derived from 
the dynamics of the continuous-time stochastic hybrid systems 
given in \cref{ssub:ctshs}
by time discretization
\cite{abate_ApproximateModelChecking_2010a,abate_ApproximateAbstractionsStochastic_2011}.
The observable $y$ of the system is defined in the same way as in the continuous-time case.

We call the transition function $T(q',x',q,x)$ $\alpha$-contractive,
if for any two distributions $F(q, x)$ and $G(q, x)$, it holds that
\begin{equation} \label{eq:dtshs_contractive}
\begin{split}
  \Big\Vert \sum_{q \in \mathcal{Q}} \int_{\mathbf{A}_q} T(q',x', q, x) 
  \big( F(q,x) - G(q,x) \big) \mathrm d x \Big\Vert \leq \alpha \nm{F(q,x) - G(q,x)}
\end{split}
\end{equation}
where $\nm{\cdot}$ is the $L_1$-norm.
This $\alpha$-contractive condition is different 
from its continuous-time counterpart
(\cref{def:cont}) in two aspects.
First, the parameter $\alpha$ of \eqref{eq:dtshs_contractive}
is the contractive factor for one discrete time step, 
while the parameter $\alpha$ of \eqref{eq:cont}
is the contractive rate for the continuous time.
Second, the contractivity of \eqref{eq:cont} is defined
with respect to the given observable, while 
the contractivity of \eqref{eq:dtshs_contractive}
is independent of the observables.
For the discrete time, the contractivity of \eqref{eq:dtshs_contractive}
generally holds for many common stochastic dynamics,
such as (discrete-time) diffusion processes.

\paragraph*{Inequality linear temporal logic (iLTL)}
We use the iLTL~\cite{04-iLTL} to capture the temporal properties of interest
for the discrete-time stochastic hybrid systems.
The iLTL can be viewed as the discrete-time version of 
the MITL introduced in \cref{sub:mitl}.
It is a variation of the common linear temporal logic~\cite{ltl2buchi2001Gastin}
by setting the atomic propositions \AP\
to be inequalities of the form $y \sim c$,
where $c \in \Rat$, 
$\sim \in \Brace{<,\leq,\geq,>}$, 
and $y$ is an observable of the system given by \eqref{eq:obs}.
(This is similar to the case of MITL in \cref{def:MITL-syn}.)
Again in the syntax of iLTL, we drop the negation operator $\neg$
by pushing it inside and using completeness of $\Brace{<,\leq,\geq,>}$.

\begin{definition}[iLTL Syntax]\label{def:iLTL-syn}
The syntax of iLTL formulas is defined using the BNF rule:
\[
\phi = \bot \mid \top \mid y \sim c \mid \phi \land \phi \mid
\phi \lor \phi
\mid \Next{\phi}
\mid \Until{\phi}{\phi}
\mid \Release{\phi}{\phi},
\]
where $c\in\Rat$ and $\sim\in\Brace{<,\leq,\geq,>}$.
\end{definition}

The discrete-time stochastic hybrid system induces a signal $f: \mathbb N \rightarrow 2^\AP$ by $(y \sim c) \in f(t)$ iff $y \sim c$ holds at time $t$.
According, we define the semantics of iLTL on the system by~\cref{def:iLTL-sem}.

\begin{definition}[iLTL Semantics]\label{def:iLTL-sem}
Let $\phi$ be an iLTL formula and $f$ be a discrete-time signal.
The satisfaction relation $\Sat$ between $f$ and $\phi$ is inductively defined according to the rules:
\begin{gather*}
\begin{array}{l@{\ \ \ }c@{\ \ \ }l}
f\Sat \bot     & {\rm iff} & \mbox{false}     \\
f\Sat \top     & {\rm iff} & \mbox{true}      \\
f\Sat y\sim c  & {\rm iff} & \Paren{y\sim c}\in f(0) \\
f\Sat \phi\lor \psi       & {\rm iff} & \Paren{f\Sat\phi}\lor \Paren{f\Sat\psi}   \\
f\Sat \phi\land\psi       & {\rm iff} & \Paren{f\Sat\phi}\land\Paren{f\Sat\psi}   \\
f\Sat \Next{\phi}         & {\rm iff} & f^1\Sat\phi \\
f\Sat \Until{\phi}{\psi}  & {\rm iff} & \exists i \in \mathbb N, (f^i\Sat\psi \land \forall j  \in [i],  f^j \Sat \phi) \\
f\Sat \Release{\phi}{\psi}& {\rm iff} & \forall i \in\mathbb N, f^i \Sat \psi \textrm{ or }
  \exists i\in\mathbb N,  (f^i\Sat\phi \land \forall j \in [i+1],  f^j \Sat \psi),
\end{array}
\end{gather*}
where $f^i(\cdot) = f(\cdot + i)$. Let $\llbracket \phi \rrbracket$ be the set of signals that satisfy $\phi$.
\end{definition}

Verifying the signals can be done by transforming them to \BAlp~\cite{ltl2buchi2001Gastin}, which can be viewed as the discrete-time version of timed automata in \cref{def:TA}.

\begin{definition}\label{def:buchi}
A \BAls\ $B$ is a tuple
  $\Paren*{\bS,\allowbreak
           \bA,\allowbreak
           \bT,\allowbreak
           \bSi,\allowbreak
           \bF}$
  where
\begin{myitems}
\item \bS\ is a finite non-empty set of {\em states},
\item \bA\ is a finite {\em alphabet},
\item $\bT\subseteq\bS\times\bA\times\bS$ is a {\em transition relation},
\item $\bSi\subseteq\bS$ is a set of {\em initial} states,
\item $\bF\subseteq\bS$ is a set of {\em final} states.
\end{myitems}
We write $s_1 \goesto [a]s_2$ instead of $\Paren{s_1,a,s_2}\in\bT$.
\end{definition}

The \BAus\ $B$ takes an infinite sequence $w \in \bA^{\pmb \omega}$ as an input and accepts it, iff there exists an infinite sequence of states $\rho \in \bS^{\pmb \omega}$ such that
\begin{inparaenum}
\item $\rho_0 \in \bSi$,
\item $\forall n \in \mathbb N, \rho_n \goesto [w_n]\rho_{n+1}$, and
\item $\Inf[\rho]\cap\bF\neq\emptyset$, where $\Inf[\rho]$ is the set of states that appear infinitely often in $\{\rho_n\}_{n=1}^\infty$.
\end{inparaenum}
%
An infinite sequence of states is called a {\em run} of $B$ if it satisfies 1 and 2, and an {\em accepting run} if it satisfies 1,2, and 3.
We define the {\em language} of $B$, denoted by $\Lang[B]$, to be the set of all infinite sequences in $\bA^{\pmb \omega}$ that are accepted by $B$.

Similar to the relation between MITL and timed automata (\cref{thm:mitl2ta}), we introduce the following result on the conversion between LTL and B\"uchi automata.

\begin{lemma}[LTL to B\"uchi automata~\cite{ltl2buchi2001Gastin, ltl-spot-2011-Duret, spot-2004-Duret}]\label{thm:ltl2ba}
For any \LTL\ formula $\phi$, a \BAls\ $B_\phi$ can be constructed such that $\Lang[B_\phi] = \llbracket \phi \rrbracket $, \ie, the set of infinite words that satisfy $\phi$ is exactly those that are accepted by $B_\phi$.
\end{lemma}

\subsection{Model reduction}

The model reduction for the discrete-time stochastic hybrid systems is similar to that for the continuous-time ones discussed in \cref{subsec:reduce_dynamics,subsec:reduce_mitl,subsec:error}, following the three steps of (i) reducing the dynamics by partitioning the state space, (ii) reducing the temporal logic specifications accordingly, and (iii) estimating the model reduction error.

\subsubsection{Reducing the Dynamics}
For a discrete-time stochastic hybrid system, we can reduce it to a finite-state Markov chain by the set-oriented method~\cite{dellnitz_approximation_1999}
which can be viewed as a discrete-time variation of 
the Mori-Zwanzig method~\cite{chorin_optimal_2000}.
Similar to~\cref{sec:modelreduction}, let $S = \{s_1,s_2,\ldots,s_n\}$ be a partition of the continuous state space $\mathbb{X}$, and $P, R$ be the corresponding projection and injection operators as given by~\eqref{eq:projection}-\eqref{eq:injection2}.
As shown in~\cref{fig:galerkin1} and~\cref{thm:discrete_model_reduction}, they induce a projection from the Markov kernel $T: m(\mathbb{X}) \rightarrow m(\mathbb{X})$ to a Markov kernel $T_r: m(S) \rightarrow m(S)$ by
\begin{equation} \label{eq:gp1}
T_r = P T R.
\end{equation}
For multiple steps, the diagram for projection is shown by the non-commutative diagram in~\cref{fig:galerkin}.

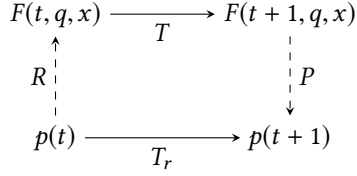
\begin{figure}[t]
\centering
\begin{tikzpicture}
  \matrix (m) [matrix of math nodes,row sep=3em,column sep=4em,minimum width=2em] {
     F(t,q,x) & F(t+1,q,x) \\
     p(t) & p(t+1) \\};
  \path[-stealth]
    (m-2-1) edge[dashed] node [left] {$R$} (m-1-1)
    (m-1-1) edge node [below] {$T$}  (m-1-2)
    (m-2-1.east|-m-2-2) edge node [below] {$T_r$} (m-2-2)
    (m-1-2) edge[dashed] node [right] {$P$} (m-2-2);
\end{tikzpicture}
\caption{Diagram for single-step reduction}  \label{fig:galerkin1}
\end{figure}

\begin{theorem} \label{thm:discrete_model_reduction}
Let $S = \{s_1,\ldots,s_n\}$ be a measurable partition of the state space $\mathbb{X}$. Then the discrete-time stochastic hybrid system reduces to a CTMC $(T_r,p_0)$ by
\begin{equation*}
p_0(i) = \int_{s_i} F(0,q,x) \d x, \quad T_r(i,j) = \int_{s_i} \int_{s_j} T(q',x',q,x) \d x' \d x.
\end{equation*}
\end{theorem}


\begin{figure}[t]
\centering
\begin{tikzpicture}
  \matrix (m) [matrix of math nodes,row sep=3em,column sep=4em,minimum width=2em] {
     F(0,q,x) & F(1,q,x) & F(t-1,q,x) & F(t,q,x) \\
     p(0) & p(1) & p(t-1) & p(t) \\};
  \path[-stealth]
    (m-2-1) edge node [left] {$R$} (m-1-1)
    (m-1-1) edge node [below] {$T$} (m-1-2)
    (m-2-1.east|-m-2-2) edge [dashed] node [below] {$T_r$} (m-2-2)
    (m-1-3) edge node [below] {$T$} (m-1-4)
    (m-2-3.east|-m-2-4)  edge [dashed] node [below] {$T_r$} (m-2-4)
    (m-1-4) edge node [right] {$P$} (m-2-4)
    (m-1-2) edge [dotted] node [below] {} (m-1-3)
    (m-2-2) edge [dotted] node [below] {} (m-2-3) ;
\end{tikzpicture}
\caption{Diagram for multiple-step reduction} \label{fig:galerkin}
\end{figure}
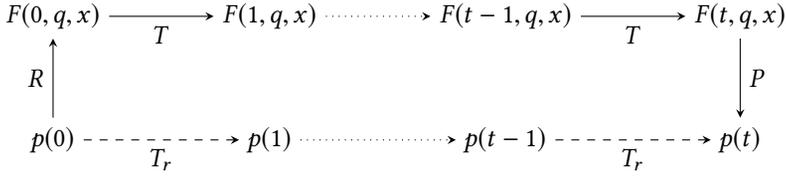

\subsubsection{Reduced iLTL}

Similar to \cref{subsec:reduce_mitl}, 
an observable $y(t)$ from \eqref{eq:obs}
for the discrete stochastic hybrid system
can be reduced approximately to an observable $y'(t)$ 
on the discrete-time Markov chain
by~\eqref{eq:rap}. 
Again, we make the assumption~\eqref{eq:ass},
as we did for the continuous-time case.
Initially, the discrepancy between 
$y(0)$ and $y'(0)$ and is given by~\eqref{riltl}.
\begin{lemma} \label{riltl}
For any $F(q,x) \in m(\mathbb{X})$ and projection operator $P$, we have
\begin{align*}
& y(0) > b + \delta_P \nm{F}_\infty \Longrightarrow y'(0) > b, \quad y'(0) > b + \delta_P \nm{F}_\infty \Longrightarrow  y(0) > b, \\
& y(0) < b - \delta_P \nm{F}_\infty \Longrightarrow y'(0) < b, \quad y'(0) < b - \delta_P \nm{F}_\infty \Longrightarrow  y(0) < b,
\end{align*}
where
\begin{equation} \label{eq:erp}
\delta_P = \tv{F(0,q,x) - R P F(0,q,x)},
\end{equation}
is the error of projection operator $P$ in total variance, 
where $\tv{\cdot}$ is the total variation distance.
\end{lemma}
%

\subsubsection{Reduction Error Estimation}

To compute the discrepancy between $y(t)$ and $y'(t)$ for any $t \in \Nat$, we first note that the projection operator $P$ is contractive.

\begin{lemma} \label{pcontra}
Let $\mathbb{S} = \{s_1,\ldots,s_n\}$ be a measurable partition of $\mathbb{X}$ and $P$ be the projection operator associated with $\mathbb{S}$. For any $F(q,x), F'(q,x) \in m(\mathbb{X})$,
\begin{equation*}
\tv{P F(q,x) - P F'(q,x)} \leq \tv{F(q,x) - F'(q,x)}.
\end{equation*}
\end{lemma}

As shown in the non-commutative diagram in~\cref{fig:galerkin}, 
the discrepancy for any $t \in \Nat$ can be written as
\begin{equation*}
\ner_t = \tv{P T^{(t)} F(0,q,x) - T_r^{(t)} P F(0,q,x)} = \tv{P T^{(t)} F(0,q,x) - P (T R P)^{(t)} F(0,q,x)}.
\end{equation*}
So, its error bound can be derived as follows.

\begin{theorem}
Given a discrete-time stochastic hybrid system and a projection operator $P$, the $t$-step ($t \geq 1$) error of projection
\begin{equation}
\ner_t \leq \sum_{i=0}^{t-1} \delta_P((T R P)^{(i)} F(0,q,x)) ,
\end{equation}
where $\delta_P$ is given in~\eqref{eq:erp}.
\end{theorem}

\begin{proof}
For $t = 1$, we have,
\begin{equation*}
\begin{split}
\ner_1 & = \tv{P T F(0,q,x) - P (T R P) F(0,q,x)} \leq \tv{T F(0,q,x) - T R P F(0,q,x)}
\\ & \leq \tv{F(0,q,x) - R P F(0,q,x)} = \delta_P(F(0,q,x)).
\end{split}
\end{equation*}
For $t > 1$, with $F$ denoting $F(0,q,x)$, we have
\begin{equation*} \label{long}
\begin{split}
\ner_t & = \tv{P T^{(t)} F - P (T R P)^{(t)} F} \leq \tv{T^{(t)} F - (T R P)^{(t)} F} \leq \tv{T^{(t)} F - T^{(t-1)} (T R P) F} \\
 \quad & + \tv{T^{(t-1)} (T R P) F - T^{(t-2)} (T R P)^{(2)} F} \ldots + \tv{T (T R P)^{(t-1)} F - (T R P)^{(t)} F} \leq \sum_{i=0}^{t-1} \delta_P((T R P)^{(i)} F).
\end{split}
\end{equation*}

\end{proof}

When $T$ is strictly contractive, we can derive a uniform error bound
for $\ner_t$ as follows.

\begin{theorem} \label{thm:eb2}
Given a discrete-time stochastic hybrid system, a projection operator $P$ and the corresponding injection $R$, if the Markov kernel $T$ is strictly contractive by factor $\alpha \in (0,1)$, then the $t$-step ($t \geq 1$) error of projection
\begin{equation}
\ner_t \leq \frac{\delta_P}{1-\alpha},
\end{equation}
where
\begin{equation} \label{deltapp}
\delta_P = \sup_{i \in \mathbb N} \delta_P((T R P)^{(i)} F(0,q,x)) .
\end{equation}
\end{theorem}

\begin{proof}
For $t = 1$, clearly $\ner_t = \delta_P$.
For $t \geq 2$, by~\eqref{long} and with $F$ denoting $F(0,q,x)$, we have
\begin{equation}
\begin{split}
\ner_t  & \leq \tv{T^{(t)} F - T^{(t-1)} (T R P) F} + \tv{T^{(t-1)} (T R P) F - T^{(t-2)} (T R P)^{(2)} F} 
\\ & \quad + \ldots + \tv{T (T R P)^{(t-1)} F - (T R P)^{(t)} F}  \leq (1 + \alpha + \ldots + \alpha^t) \delta_P \leq \frac{\delta_P}{1-\alpha}.
\end{split}
\end{equation}
\end{proof}

By combining~\cref{riltl} and~\cref{thm:eb2}, we can derive the following theorem on the relationship between linear inequalities on the original Markov process and linear inequalities on the reduced Markov process.

\begin{theorem} \label{relation}
Given a measurable partition $\mathbb{S} = \{s_1,\ldots,s_n\}$ and the corresponding projection operator $P$, a discrete-time stochastic hybrid system and its reduction $(T_r,p_0)$ satisfies the equations:
\begin{align}
& y (t) > b + \frac{\delta_P \nm{F}_\infty}{1-\alpha}  \Longrightarrow y' (t) > b, \quad y' (t) > b + \frac{\delta_P \nm{F}_\infty}{1-\alpha}  \Longrightarrow  y (t) > b, \\
& y (t) < b - \frac{\delta_P \nm{F}_\infty}{1-\alpha}  \Longrightarrow y' (t) < b, \quad y' (t) < b - \frac{\delta_P \nm{F}_\infty}{1-\alpha}  \Longrightarrow  y (t) < b,
\end{align}
for any $t \geq 0$, where $\delta_p$ is given by~\eqref{deltapp} respectively.
\end{theorem}

\cref{relation} can be viewed as the discrete-time counterpart of \cref{thm:ctshs_main}.
In \cref{thm:ctshs_main}, the model reduction error is bounded by two term: one for the initial error, and the other for the error accumulated over time.
In \cref{relation}, these two terms are combined into one, due to the difference between the contractivity condition \eqref{eq:dtshs_contractive} and \eqref{eq:cont}.

Following \cref{relation}, to verify an iLTL formula $\phi$ for an $\alpha$-contractive discrete-time stochastic hybrid system introduced in \cref{ssub:ctshs}, we can strengthen $\phi$ to $\psi$ by replacing the atomic propositions according to \cref{relation}.
If $\psi$ holds for the DTMC derived from the discrete-time stochastic hybrid system following the aforementioned model reduction procedure, then $\phi$ holds for the discrete-time stochastic hybrid system.

\subsection{Statistical Model Checking of iLTL} 
\label{sec:statistical_model_checking_of_iltl}

Similar to \cref{sec:mc}, we introduce a statistical model checking procedure for iLTL specifications on the reduced systems.
Again, we denote the atomic proposition $\mY = \sum_{i=1}^n \mR_i \mY_i = \mR \cdot \mY > \mC$ by a pair $\Paren{\mR,\mC}$.
For an iLTL formula $\varphi$ and a discrete-time Markov chain generating a sequence of distributions $w  =  p_0p_1p_2\ldots$, define $u = u_0 u_1 u_2 \ldots$ where $u_t = \Brace{(\mR,\mC) \in \AP_\varphi \mid \mR\cdot p_t>\mC}$ is the set of atomic propositions that are true at time $t$.
Similar to \Cref{sec:mc}, our algorithm in this section has four steps:
\begin{itemize}
\item Construct the \BAlp\ $B_\varphi$ and $B_{\neg\varphi}$.
\item Find a time step $T$ at which $\mY(T)$ is very close to our estimation of the
      invariant distribution.
\item Construct $B_{M,\varphi}$,
\item If $\Lang[B_{M,\varphi}]\cap\Sem{B_\varphi}=\emptyset$ then return \No, 
      if $\Lang[B_{M,\varphi}]\cap\Sem{B_{\neg\varphi}}=\emptyset$ then return \Yes,
      otherwise, return \Unknown. 
\end{itemize}

These steps are similar to their corresponding step in \Cref{sec:mc}.
For example, the first step is carried out using \Cref{thm:ltl2ba}.
Simulation of discrete and continuous Markov chains are different procedures, but they both can be performed efficiently, and that is what we need for the second and third steps.
Similarly, checking emptiness of intersection of \TAlp\ and \BAlp\ are different procedures, but they are both known to be decidable~\cite{buchi}.
The main difference with \Cref{fig:alg-part2-ct} is that 
since in \Cref{thm:ltl2ba} time is discrete, 
to find labels of $B_{M,\varphi}$, we only run one instance of $\ALG_0$ at each step.
\Cref{alg:iLTL} shows the pseudocode for different steps.
Again, similar to \Cref{fig:alg-part2-ct}, \Unknown\ labels are modeled using two locations; 
one labeled by $\Brace{(y\sim c)}$ 
and the other labeled by $\emptyset$. 
However, since the time is discrete for \BAlp, 
there will be no extra transition between these two locations.

Similar to our previous algorithm, 
in addition to 
  a Markov chain $M$,
  iLTL formula $\varphi$, and
  $\mY^*$, an estimation of the invariant distribution $\inv{\mY}$,
\Cref{alg:iLTL} takes
  two error parameters $\alpha,\gamma\in(0,1)$ and 
  two indifference parameters $\delta,\delta'\in(0,1)$.
The parameters $\delta'$ and $\frac{1}{2}\min\Brace{\alpha,\gamma}$ 
are used to find the time bound $T$, 
and the parameters $\delta$, $\frac{1}{2}\alpha$, and $\frac{1}{2}\gamma$ are used to construct labels of
\BAls\ $B_{M,\varphi}$ before reaching step $T$.
We have the following guarantee about the algorithm:
\begin{align*}
  \Prb[\mathtt{res}=\No\ \,  \mid M \Sat \varphi] \leq \alpha, \quad \Prb[\mathtt{res}=\Yes \mid M \nSat\varphi] \leq \alpha,
  \\
  \big(\forall \sigma\oftype B^{\delta}(\mR\cdot\mY), \ \sigma\Sat\varphi\big){\implies}
  \Prb[\mathtt{res}{=}\Unknown] \leq \alpha + \gamma,
  \\
  \big(\forall \sigma\oftype B^{\delta}(\mR\cdot\mY), \ \sigma\nSat\varphi\big){\implies}
  \Prb[\mathtt{res}{=}\Unknown] \leq \alpha + \gamma,
\end{align*}
where $B^\delta(r\cdot p)$ is the tube of discrete functions that are $\delta$-close to $r\cdot p$.

\begin{algorithm}[t]
\KwData{%
  Markov chain $(M,\mY_0)$,
  estimation of invariant distribution $\mY^*$,
  iLTL formula $\varphi$,
  parameters $\alpha$, $\gamma$, $\delta$, $\delta'$}
\KwResult{\Yes, \No, or \Unknown}
\Function{$\mathtt{NumberOfSamplingSteps}()$}{%
  $t\gets1$\\
  $\alpha'\gets\frac{1}{2}\min\Brace{\alpha,\gamma}$\\
  \While{$\ALGClose\Paren*{\mY(t), \mY^*, \frac{1}{2}\alpha', \frac{\delta'}{3}}= \mathtt{failed}$}{%
    $t \gets 2\times t$\\
    $\alpha'\gets\frac{1}{2}\alpha'$
  }
  \Return{t+1}
}
\Function{$\mathtt{LabelFiniteNumberOfSteps}(m \in \mathbb N)$}{%
  \ForAll{$t \in \Brace{0,1,\ldots,m-1}$, $(\mR,\mC) \in \AP$}{%
    $asg\Paren*{t,(\mR,\mC)}\gets\ALG_0(\mR\cdot\mY(t),\mC,\frac{\alpha}{2m\Size*{\AP}},\frac{\gamma}{2m\Size*{\AP}},\frac{\delta}{3})$
  }
  \Return{$asg$}
}
\Function{$\mathtt{AddLabelsOfInvariantDistribution}(m \in \mathbb N, asg \in \mathbb N\times\AP \rightarrow \Brace{\Yes,\No,\Unknown})$}{%
  \ForAll{$t \in \Brace{m,m+1,\ldots}$, $(\mR,\mC) \in \AP$}{%
    \eIf{$\mR\cdot p^*>\mC$}{$asg(t,(\mR,\mC))\gets\Yes$}{$asg(t,(\mR,\mC))\gets\No$}
  }
  \Return{$asg$}
}
\Function{$\mathtt{ModelCheck}$}{%
  $T   \gets \mathtt{NumberOfSamplingSteps}()$\;
  $asg \gets \mathtt{LabelFiniteNumberOfSteps}(T)$\;
  $asg \gets \mathtt{AddLabelsOfInvariantDistribution}(T,asg)$\;
  $\llbracket asg \rrbracket \gets$ the \BAls\ that accepts exactly the set of infinite paths induced by $asg$\\
  \uIf{$\Lang(B_\varphi)      \cap\Lang[\llbracket asg \rrbracket] = \emptyset$}{\Return \No}
  \uIf{$\Lang(B_{\neg\varphi})\cap\Lang[\llbracket asg \rrbracket] = \emptyset$}{\Return \Yes}
  {\Return \Unknown} 
}
\caption{Model checking Markov chains against iLTL formulas}\label{alg:iLTL}
\end{algorithm}

\section{Case Study} \label{sec:case}

We implemented the proposed model reduction and statistical verification algorithm 
on high-dimensional stochastic hybrid systems 
with polynomial dynamics for the continuous states 
to demonstrate the scalability.
In this section we present our experimental results.
Consider a piecewise linear jump system under nonlinear perturbation with the continuous state $x(t) \in \Real^n$ and the discrete state $q(t) \in [m]$ with $m \in \Nat$.
The continuous dynamics is
\begin{equation} \label{eq:ex_linear_system}
	\frac{\d x}{\d t} = (A_{q(t)} + c_{q(t)} \nm{x(t)}_\infty) x(t)
\end{equation}
where $A_i \in \Real^{n \times n}$ is Hurwitz and $c_i > 0$ for $i \in [m]$.
The discrete state jumps spontaneously 
with the rate $\lambda_1$ from $j$ to $j-1$ 
for $j = 2, \ldots, m$ 
and with the rate $\lambda_2$ from $j$ to $j+1$ 
for $j = 1, \ldots, m-1$.
%
%
Initially, the continuous state is distributed uniformly on the hypercube $C = \set{x(0) \in \Real^n}{\nm{x(t)}_\infty \leq K}$; and the discrete state $q(0)$ uniformly on $[m]$.

Assume that the elements of the dynamical matrices $A_i$ are non-positive, 
then $x(t) \in C$ for all $t \in \Real$.
Therefore, we can partition the state space into $(2 \eta)^n \times m$, each of length $1/\eta$.
The hypercubes are indexed by $(i_1, \ldots, i_n, j)$ with $\abs{i_k} \in \{-\eta, \ldots, -1,1, \ldots, \eta\}$, $j \in [m]$, and $k \in [n]$.
The transition probability rates are zero except
\begin{equation*}
\begin{split}
	& \lambda((i_1, \ldots, i_n, j) \rightarrow (i_1, \ldots, i_n, j-1)) =  \lambda_1,
	\quad \lambda((i_1, \ldots, i_n, j) \rightarrow (i_1, \ldots, i_n, j+1)) =  \lambda_2,
	\\ & \lambda((i_1, \ldots, i_k+1, \ldots, i_n, j) \rightarrow (i_1, \ldots, i_k, \ldots, i_n, j))
	= c_j K \max_k \frac{\abs{i_k}}{\eta^3} + \int_S \frac{(A_j x)_k}{\eta^2} \d x_1 \ldots \d x_{k-1} \d x_{k+1} \ldots \d x_n,
\end{split}
\end{equation*}
where the reduction error $\Theta_y(t)$ in \eqref{eq:uniform eb} is less than $0.1$ for all $t$.
The desired property is 
$$\Until{\top}{\big(w(F(t,q,x)){>}p\big)}[[0,T]],$$
where 
$T$ is a time bound (could be $\infty$),
$p$ is a probability threshold, and
$w(\cdot)$ is the indicator function on a non-convex predicate stating 
exactly two elements of the continuous state are more than $\lceil K/2 \rceil$ away from the origin
(formally, the predicate holds for a continuous state $x$ \iFF\
$\Size{\Brace{i\in[n]\mid |x_i|\geq\lceil K/2 \rceil}}=2$).
It asserts that before time $T$, a probability distribution will be reached such that
the probability of a state $x$ in that distribution satisfying the aforementioned predicate is larger than $p$.
%
%

\newcommand{\DefineVS}[2]{\expandafter\providecommand\csname VS#1\endcsname{#2}}%
\newcommand{\DefineES}[2]{\expandafter\providecommand\csname ES#1\endcsname{#2}}%
\newcommand{\DefineNS}[2]{\expandafter\providecommand\csname NS#1\endcsname{#2}}%
\newcommand{\VS}[1]{\csname VS#1\endcsname}%
\newcommand{\ES}[1]{\csname ES#1\endcsname}%
\newcommand{\NS}[1]{\csname NS#1\endcsname}%

\newcommand{\Coef}{0.27718585822}
\newcommand{\Vscale}{1}%
\newcommand{\ChartType}{T}%
\newcommand{\FirstUnbounded}{F}%
\newcommand{\YName}{Name Not Given}%
\newcommand{\SmallTimes}{\mbox{\tiny$\times$}}%

\NewDocumentCommand{\AddErrorBar}{O{black}mmm}{%
\draw[color=#1] ($(#2)+(0.125,#3*\Vscale+#4*\Coef*\Vscale)-(0.05,0)$) -- ++(0.1,0);
\draw[color=#1] ($(#2)+(0.125,#3*\Vscale-#4*\Coef*\Vscale)-(0.05,0)$) -- ++(0.1,0);
\draw[color=#1] ($(#2)+(0.175,#3*\Vscale-#4*\Coef*\Vscale)-(0.05,0)$) -- 
                ($(#2)+(0.175,#3*\Vscale+#4*\Coef*\Vscale)-(0.05,0)$);}%

\npdecimalsign{.}%
\newcommand{\NDigits}{1}

\NewDocumentCommand{\ToDuration}{m}{%
\ifdim #1pt < 60pt 
\numprint{#1}\,s
\else
\pgfmathparse{int(floor(div(#1,60)))}\pgfmathresult\,m
\pgfmathparse{mod(#1,60)}\numprint{\pgfmathresult}\,s
\fi}%

\NewDocumentCommand{\ToKNumber}{m}{\numprint{#1}\,k}%
  
\NewDocumentCommand{\AddValue}{mmO{0}}{%
  \node[anchor=south] at ($(#1)+(0.125,#2*\Vscale+#3*\Coef*\Vscale)$)
  {\scalebox{0.7}
    {\rotatebox{90}
    {\ToDuration{#2}}}}; }%

\definecolor{Cbar1}{RGB}{073,141,133}%
\definecolor{Cbar2}{RGB}{118,069,039}%
\definecolor{Cbar3}{RGB}{130,020,130}%
\definecolor{Cbar4}{RGB}{090,140,065}%
\definecolor{Cbar5}{RGB}{040,040,140}%
\definecolor{Cbar6}{RGB}{070,116,193}

\newcommand{\DrawGraph}{%
\begin{tikzpicture}
\coordinate (b1)  at (0.075,0);
\coordinate (b2)  at ($(b1)+(0.30,0)$);
\coordinate (b3)  at ($(b2)+(0.30,0)$);
\coordinate (b4)  at ($(b3)+(0.50,0)$);
\coordinate (b5)  at ($(b4)+(0.30,0)$);
\coordinate (b6)  at ($(b5)+(0.30,0)$);
\coordinate (b7)  at ($(b6)+(0.50,0)$);
\coordinate (b8)  at ($(b7)+(0.30,0)$);
\coordinate (b9)  at ($(b8)+(0.30,0)$);
\coordinate (b10) at ($(b9)+(0.50,0)$);
\coordinate (b11) at ($(b10)+(0.30,0)$);
\coordinate (b12) at ($(b11)+(0.30,0)$);
\coordinate (b13) at ($(b12)+(0.50,0)$);
\coordinate (b14) at ($(b13)+(0.30,0)$);
\coordinate (b15) at ($(b14)+(0.30,0)$);
\coordinate (b16) at ($(b15)+(0.50,0)$);
\coordinate (b17) at ($(b16)+(0.30,0)$);
\coordinate (b18) at ($(b17)+(0.30,0)$);

\ifthenelse{\equal{\ChartType}{U}}
{\draw[draw=white,fill=gray!10]  (-0.1,0) rectangle (4.4,4);}
{\draw[draw=white,fill=gray!10]  (-0.1,0) rectangle (6.6,4);}

\node[anchor=north] at ($(-0.7, 0.00)+(b1)$) {\scriptsize threshold};
\node[anchor=north] at ($(-0.7,-0.30)+(b1)$) {\scriptsize \#states};
\ifthenelse{\equal{\FirstUnbounded}{T}}{%
\node[anchor=north,color=white] at ($(-0.7,-0.65)+(b1)$) {\scriptsize \#checks};
}{%
\node[anchor=north] at ($(-0.7,-0.65)+(b1)$) {\scriptsize \#checks};
}%

\ifthenelse{\equal{\FirstUnbounded}{T}}%
{\newcommand{\LowestY}{-0.70}}%
{\newcommand{\LowestY}{-1.05}}%

\ifthenelse{\equal{\ChartType}{U}}{%
\draw[color=gray!30]  ($(-1.15, 0.00)$) rectangle ($(4.4,\LowestY)$);
\draw[color=gray!30]  ($(-1.15,-0.35)$) --        ($(4.4,-0.35)$);
\draw[color=gray!30]  ($(-1.15,-0.70)$) --        ($(4.4,-0.70)$);
}{%
\draw[color=gray!30]  ($(-1.15, 0.00)$) rectangle ($(6.6,\LowestY)$);
\draw[color=gray!30]  ($(-1.15,-0.35)$) --        ($(6.6,-0.35)$);
\draw[color=gray!30]  ($(-1.15,-0.70)$) --        ($(6.6,-0.70)$);
}%

\draw[color=gray!30]  ($      (-0.10, 0.00)$) --  ($      (-0.10,\LowestY)$);
\draw[color=gray!30]  ($(b4) +(-0.12, 0.00)$) --  ($(b4) +(-0.12,\LowestY)$);
\draw[color=gray!30]  ($(b7) +(-0.12, 0.00)$) --  ($(b7) +(-0.12,\LowestY)$);
\draw[color=gray!30]  ($(b10)+(-0.12, 0.00)$) --  ($(b10)+(-0.12,\LowestY)$);
\ifthenelse{\equal{\ChartType}{U}}{}{%
\draw[color=gray!30]  ($(b13)+(-0.12, 0.00)$) --  ($(b13)+(-0.12,\LowestY)$);
\draw[color=gray!30]  ($(b16)+(-0.12, 0.00)$) --  ($(b16)+(-0.12,\LowestY)$);
}%

\node[anchor=north] at ($(0.4,-0.30)+(b1)$)  {\scriptsize $\ 4\SmallTimes20^{5}$};
\node[anchor=north] at ($(0.4,-0.30)+(b4)$)  {\scriptsize $\ 4\SmallTimes20^{10}$};
\node[anchor=north] at ($(0.4,-0.30)+(b7)$)  {\scriptsize $\ 4\SmallTimes20^{15}$};
\node[anchor=north] at ($(0.4,-0.30)+(b10)$) {\scriptsize $\ 4\SmallTimes20^{20}$};
\ifthenelse{\equal{\ChartType}{U}}{}{%
\node[anchor=north] at ($(0.4,-0.30)+(b13)$) {\scriptsize $\ 4\SmallTimes20^{30}$};
\node[anchor=north] at ($(0.4,-0.30)+(b16)$) {\scriptsize $\ 4\SmallTimes20^{40}$};
}%

\ifthenelse{\equal{\FirstUnbounded}{T}}{}{%
\nprounddigits{0}
\node[anchor=north] at ($(0.45,-0.68)+(b1)$)  {\scalebox{0.65}{$\numprint{\NS{1}}$}};
\node[anchor=north] at ($(0.45,-0.68)+(b4)$)  {\scalebox{0.65}{$\numprint{\NS{2}}$}};
\node[anchor=north] at ($(0.45,-0.68)+(b7)$)  {\scalebox{0.65}{$\numprint{\NS{3}}$}};
\node[anchor=north] at ($(0.45,-0.68)+(b10)$) {\scalebox{0.65}{$\numprint{\NS{4}}$}};
\ifthenelse{\equal{\ChartType}{U}}{}{%
\node[anchor=north] at ($(0.45,-0.68)+(b13)$) {\scalebox{0.65}{$\numprint{\NS{5}}$}};
\node[anchor=north] at ($(0.45,-0.68)+(b16)$) {\scalebox{0.65}{$\numprint{\NS{6}}$}};
}%
\nprounddigits{\NDigits}
}%

\node[anchor=north] at ($(0.1,0)+(b1)$) {\scriptsize $.2$};
\node[anchor=north] at ($(0.1,0)+(b2)$) {\scriptsize $.5$};
\node[anchor=north] at ($(0.1,0)+(b3)$) {\scriptsize $.8$};

\node[anchor=north] at ($(0.1,0)+(b4)$) {\scriptsize $.2$};
\node[anchor=north] at ($(0.1,0)+(b5)$) {\scriptsize $.5$};
\node[anchor=north] at ($(0.1,0)+(b6)$) {\scriptsize $.8$};

\node[anchor=north] at ($(0.1,0)+(b7)$) {\scriptsize $.2$};
\node[anchor=north] at ($(0.1,0)+(b8)$) {\scriptsize $.5$};
\node[anchor=north] at ($(0.1,0)+(b9)$) {\scriptsize $.8$};

\node[anchor=north] at ($(0.1,0)+(b10)$) {\scriptsize $.2$};
\node[anchor=north] at ($(0.1,0)+(b11)$) {\scriptsize $.5$};
\node[anchor=north] at ($(0.1,0)+(b12)$) {\scriptsize $.8$};

\ifthenelse{\equal{\ChartType}{U}}{}{%
\node[anchor=north] at ($(0.1,0)+(b13)$) {\scriptsize $.2$};
\node[anchor=north] at ($(0.1,0)+(b14)$) {\scriptsize $.5$};
\node[anchor=north] at ($(0.1,0)+(b15)$) {\scriptsize $.8$};

\node[anchor=north] at ($(0.1,0)+(b16)$) {\scriptsize $.2$};
\node[anchor=north] at ($(0.1,0)+(b17)$) {\scriptsize $.5$};
\node[anchor=north] at ($(0.1,0)+(b18)$) {\scriptsize $.8$};
}%

\draw[color=Cbar1,fill=Cbar1!50] (b1) rectangle ++($(0.25,\Vscale*\VS{1})$);
\draw[color=Cbar1,fill=Cbar1!50] (b2) rectangle ++($(0.25,\Vscale*\VS{2})$);
\draw[color=Cbar1,fill=Cbar1!50] (b3) rectangle ++($(0.25,\Vscale*\VS{3})$);

\draw[color=Cbar2,fill=Cbar2!50] (b4) rectangle ++($(0.25,\Vscale*\VS{4})$);
\draw[color=Cbar2,fill=Cbar2!50] (b5) rectangle ++($(0.25,\Vscale*\VS{5})$);
\draw[color=Cbar2,fill=Cbar2!50] (b6) rectangle ++($(0.25,\Vscale*\VS{6})$);

\draw[color=Cbar3,fill=Cbar3!50] (b7) rectangle ++($(0.25,\Vscale*\VS{7})$);
\draw[color=Cbar3,fill=Cbar3!50] (b8) rectangle ++($(0.25,\Vscale*\VS{8})$);
\draw[color=Cbar3,fill=Cbar3!50] (b9) rectangle ++($(0.25,\Vscale*\VS{9})$);

\draw[color=Cbar4,fill=Cbar4!50] (b10) rectangle ++($(0.25,\Vscale*\VS{10})$);
\draw[color=Cbar4,fill=Cbar4!50] (b11) rectangle ++($(0.25,\Vscale*\VS{11})$);
\draw[color=Cbar4,fill=Cbar4!50] (b12) rectangle ++($(0.25,\Vscale*\VS{12})$);

\ifthenelse{\equal{\ChartType}{U}}{}{%
\draw[color=Cbar5,fill=Cbar5!50] (b13) rectangle ++($(0.25,\Vscale*\VS{13})$);
\draw[color=Cbar5,fill=Cbar5!50] (b14) rectangle ++($(0.25,\Vscale*\VS{14})$);
\draw[color=Cbar5,fill=Cbar5!50] (b15) rectangle ++($(0.25,\Vscale*\VS{15})$);

\draw[color=Cbar6,fill=Cbar6!50] (b16) rectangle ++($(0.25,\Vscale*\VS{16})$);
\draw[color=Cbar6,fill=Cbar6!50] (b17) rectangle ++($(0.25,\Vscale*\VS{17})$);
\draw[color=Cbar6,fill=Cbar6!50] (b18) rectangle ++($(0.25,\Vscale*\VS{18})$);
}%

\AddErrorBar[Cbar1]{b1}{\VS{1}}{\ES{1}};
\AddErrorBar[Cbar1]{b2}{\VS{2}}{\ES{2}};
\AddErrorBar[Cbar1]{b3}{\VS{3}}{\ES{3}};

\AddErrorBar[Cbar2]{b4}{\VS{4}}{\ES{4}};
\AddErrorBar[Cbar2]{b5}{\VS{5}}{\ES{5}};
\AddErrorBar[Cbar2]{b6}{\VS{6}}{\ES{6}};

\AddErrorBar[Cbar3]{b7}{\VS{7}}{\ES{7}};
\AddErrorBar[Cbar3]{b8}{\VS{8}}{\ES{8}};
\AddErrorBar[Cbar3]{b9}{\VS{9}}{\ES{9}};

\AddErrorBar[Cbar4]{b10}{\VS{10}}{\ES{10}};
\AddErrorBar[Cbar4]{b11}{\VS{11}}{\ES{11}};
\AddErrorBar[Cbar4]{b12}{\VS{12}}{\ES{12}};

\ifthenelse{\equal{\ChartType}{U}}{}{%
\AddErrorBar[Cbar5]{b13}{\VS{13}}{\ES{13}};
\AddErrorBar[Cbar5]{b14}{\VS{14}}{\ES{14}};
\AddErrorBar[Cbar5]{b15}{\VS{15}}{\ES{15}};

\AddErrorBar[Cbar6]{b16}{\VS{16}}{\ES{16}};
\AddErrorBar[Cbar6]{b17}{\VS{17}}{\ES{17}};
\AddErrorBar[Cbar6]{b18}{\VS{18}}{\ES{18}};
}%

\AddValue{b1}{\VS{1}}[\ES{1}];
\AddValue{b2}{\VS{2}}[\ES{2}];
\AddValue{b3}{\VS{3}}[\ES{3}];

\AddValue{b4}{\VS{4}}[\ES{4}];
\AddValue{b5}{\VS{5}}[\ES{5}];
\AddValue{b6}{\VS{6}}[\ES{6}];

\AddValue{b7}{\VS{7}}[\ES{7}];
\AddValue{b8}{\VS{8}}[\ES{8}];
\AddValue{b9}{\VS{9}}[\ES{9}];

\AddValue{b10}{\VS{10}}[\ES{10}];
\AddValue{b11}{\VS{11}}[\ES{11}];
\AddValue{b12}{\VS{12}}[\ES{12}];

\ifthenelse{\equal{\ChartType}{U}}{}{%
\AddValue{b13}{\VS{13}}[\ES{13}];
\AddValue{b14}{\VS{14}}[\ES{14}];
\AddValue{b15}{\VS{15}}[\ES{15}];

\AddValue{b16}{\VS{16}}[\ES{16}];
\AddValue{b17}{\VS{17}}[\ES{17}];
\AddValue{b18}{\VS{18}}[\ES{18}];
}%

\node[rotate=90] at (-0.3,2) {\footnotesize\YName};

\ifthenelse{\equal{\ChartType}{U}}
{\draw  (-0.1,0) rectangle (4.4,4);}
{\draw  (-0.1,0) rectangle (6.6,4);}

\end{tikzpicture}}%

\newcommand{\DataBounded}{%
\renewcommand{\ChartType}{B}%
\renewcommand{\FirstUnbounded}{F}%

\DefineNS{1} {520383}%
\DefineNS{2} {627919}%
\DefineNS{3} {931853}%
\DefineNS{4}{1213187}%
\DefineNS{5}{1833850}%
\DefineNS{6}{2421269}%

\DefineVS{1}{3.3}\DefineES{1}{0.62}%
\DefineVS{2}{5.2}\DefineES{2}{1.25}%
\DefineVS{3}{1.0}\DefineES{3}{0.11}%

\DefineVS{4}{2.1}\DefineES{4}{0.52}%
\DefineVS{5}{0.7}\DefineES{5}{0.04}%
\DefineVS{6}{0.3}\DefineES{6}{0.02}%

\DefineVS{7}{4.5}\DefineES{7}{0.39}%
\DefineVS{8}{2.6}\DefineES{8}{0.11}%
\DefineVS{9}{1.1}\DefineES{9}{0.06}%

\DefineVS{10}{2.8}\DefineES{10}{0.19}%
\DefineVS{11}{1.8}\DefineES{11}{0.12}%
\DefineVS{12}{0.8}\DefineES{12}{0.03}%

\DefineVS{13}{7.9}\DefineES{13}{0.46}%
\DefineVS{14}{4.9}\DefineES{14}{0.24}%
\DefineVS{15}{2.1}\DefineES{15}{0.09}%

\DefineVS{16}{19.3}\DefineES{16}{0.42}%
\DefineVS{17}{12.0}\DefineES{17}{0.34}%
\DefineVS{18}{5.2}\DefineES{18}{0.17}%

\renewcommand{\Vscale}{0.165}%
\renewcommand{\YName}{Average Time}%
\renewcommand{\NDigits}{1}%
}%
\newcommand{\DataBoundedTen}{%
\renewcommand{\ChartType}{B}%
\renewcommand{\FirstUnbounded}{F}%

\DefineNS{1} {5203828}%
\DefineNS{2} {6279181}%
\DefineNS{3} {9318530}%
\DefineNS{4}{12131863}%
\DefineNS{5}{18338499}%
\DefineNS{6}{24212686}%

\DefineVS{1}{36.9}\DefineES{1}{7.31}%
\DefineVS{2}{89.2}\DefineES{2}{16.73}%
\DefineVS{3}{15.9}\DefineES{3}{3.08}%

\DefineVS{4}{18.2}\DefineES{4}{4.17}%
\DefineVS{5}{25.2}\DefineES{5}{5.18}%
\DefineVS{6}{4.9}\DefineES{6}{0.72}%

\DefineVS{7}{117.7}\DefineES{7}{24.59}%
\DefineVS{8}{131.7}\DefineES{8}{18.17}%
\DefineVS{9}{39.1}\DefineES{9}{4.52}%

\DefineVS{10}{71.2}\DefineES{10}{15.20}%
\DefineVS{11}{67.9}\DefineES{11}{10.51}%
\DefineVS{12}{15.5}\DefineES{12}{2.50}%

\DefineVS{13}{151.9}\DefineES{13}{29.12}%
\DefineVS{14}{121.3}\DefineES{14}{14.90}%
\DefineVS{15}{32.6}\DefineES{15}{3.95}%

\DefineVS{16}{283.2}\DefineES{16}{46.11}%
\DefineVS{17}{220.3}\DefineES{17}{27.78}%
\DefineVS{18}{55.5}\DefineES{18}{5.76}%

\renewcommand{\Vscale}{0.01}%
\renewcommand{\YName}{Average Time}%
\renewcommand{\NDigits}{0}%
}%
\newcommand{\DataUnboundedH}{%
\renewcommand{\ChartType}{U}%
\renewcommand{\FirstUnbounded}{T}%

\DefineNS{1}{16716254}%
\DefineNS{2}{24632599}%
\DefineNS{3}{41622146}%
\DefineNS{4}{58120115}%

\DefineVS{1}{9.9}\DefineES{1}{5.79}%
\DefineVS{2}{7.9}\DefineES{2}{5.05}%
\DefineVS{3}{12.0}\DefineES{3}{9.35}%

\DefineVS{4}{25.5}\DefineES{4}{17.07}%
\DefineVS{5}{27.2}\DefineES{5}{10.95}%
\DefineVS{6}{28.5}\DefineES{6}{17.76}%

\DefineVS{7}{111.7}\DefineES{7}{53.43}%
\DefineVS{8}{95.5}\DefineES{8}{44.13}%
\DefineVS{9}{108.1}\DefineES{9}{51.82}%

\DefineVS{10}{221.6}\DefineES{10}{118.34}%
\DefineVS{11}{232.9}\DefineES{11}{101.86}%
\DefineVS{12}{219.1}\DefineES{12}{97.40}%

\renewcommand{\Vscale}{0.0049}%
\renewcommand{\YName}{Average Time}%
\renewcommand{\NDigits}{0}%
}%
\newcommand{\DataUnboundedR}{%
\renewcommand{\ChartType}{U}%
\renewcommand{\FirstUnbounded}{F}%

\DefineNS{1}{16716254}%
\DefineNS{2}{24632599}%
\DefineNS{3}{41622146}%
\DefineNS{4}{58120115}%

\DefineVS{1}{165.9}\DefineES{1}{38.04}%
\DefineVS{2}{356.8}\DefineES{2}{96.45}%
\DefineVS{3}{76.9}\DefineES{3}{25.93}%

\DefineVS{4}{95.5}\DefineES{4}{36.54}%
\DefineVS{5}{129.5}\DefineES{5}{48.65}%
\DefineVS{6}{23.5}\DefineES{6}{9.56}%

\DefineVS{7}{200.6}\DefineES{7}{53.96}%
\DefineVS{8}{221.2}\DefineES{8}{49.72}%
\DefineVS{9}{50.9}\DefineES{9}{18.94}%

\DefineVS{10}{338.9}\DefineES{10}{79.04}%
\DefineVS{11}{305.7}\DefineES{11}{57.83}%
\DefineVS{12}{80.9}\DefineES{12}{11.84}%

\renewcommand{\Vscale}{0.0049}%
\renewcommand{\YName}{Average Time}%
\renewcommand{\NDigits}{0}%
}%
\newcommand{\DataUnboundedT}{%
\renewcommand{\ChartType}{U}%
\renewcommand{\FirstUnbounded}{F}%

\DefineNS{1}{16716254}%
\DefineNS{2}{24632599}%
\DefineNS{3}{41622146}%
\DefineNS{4}{58120115}%

\DefineVS{1}{175.9}\DefineES{1}{43.83}%
\DefineVS{2}{364.7}\DefineES{2}{101.50}%
\DefineVS{3}{88.9}\DefineES{3}{35.28}%

\DefineVS{4}{121.0}\DefineES{4}{53.61}%
\DefineVS{5}{156.7}\DefineES{5}{59.60}%
\DefineVS{6}{52.0}\DefineES{6}{27.32}%

\DefineVS{7}{312.2}\DefineES{7}{107.39}%
\DefineVS{8}{316.7}\DefineES{8}{93.85}%
\DefineVS{9}{159.0}\DefineES{9}{70.76}%

\DefineVS{10}{560.5}\DefineES{10}{197.38}%
\DefineVS{11}{538.6}\DefineES{11}{159.69}%
\DefineVS{12}{300.1}\DefineES{12}{109.24}%

\renewcommand{\Vscale}{0.0049}%
\renewcommand{\YName}{Average Time}%
\renewcommand{\NDigits}{0}%
}%

\begin{figure}[t]
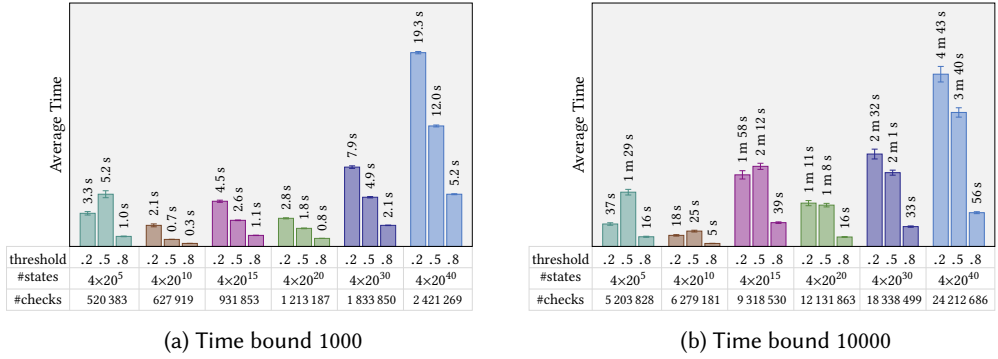

\centering
\begin{subfigure}[b]{0.5\textwidth}%
\DataBounded\scalebox{0.8}{\DrawGraph}%
\caption{Time bound $1000$}\label{fig:1000}
\end{subfigure}%
\begin{subfigure}[b]{0.5\textwidth}%
\DataBoundedTen\scalebox{0.8}{\DrawGraph}%
\caption{Time bound $10000$}\label{fig:10000}
\end{subfigure}%
\caption{Bounded Time}%
\label{fig:bounded}
\end{figure}

\begin{figure}[t]
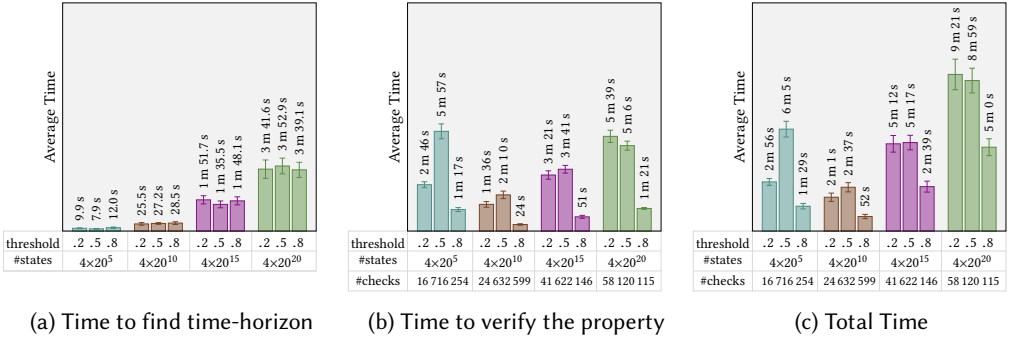

\centering
\begin{subfigure}[b]{0.33\textwidth}%
\DataUnboundedH\scalebox{0.75}{\DrawGraph}%
\caption{Time to find time-horizon}\label{fig:infH}
\end{subfigure}%
\begin{subfigure}[b]{0.33\textwidth}%
\DataUnboundedR\scalebox{0.75}{\DrawGraph}%
\caption{Time to verify the property}\label{fig:infR}
\end{subfigure}%
\begin{subfigure}[b]{0.33\textwidth}%
\DataUnboundedT\scalebox{0.75}{\DrawGraph}%
\caption{Total Time}\label{fig:infT}
\end{subfigure}%
\caption{Unbounded Time}%
\label{fig:unbounded}
\end{figure}

We ran Algorithm~\ref{fig:alg-part2-ct} on multiple instances of this problem.
In all of our experiments, 
  $\lambda_1=0.03$,
  $\lambda_2=0.02$,
  $K=1$,
  $\eta=10$, and
  $\alpha=\beta=\delta_1=0.1$.
We also fixed the number of discrete states ($m$) to be $4$. 
The dimension of the continuous state is chosen from 
$\Brace{5,10,15,20,30,40}$.
These settings result in CTMCs with a large number of states: the smallest example has $1.28{\times}10^{7}$ states, and the largest example has more than $4.39{\times}10^{52}$ states.
In all the experiments, we set $c_1=0.1$, $c_2=0.2$, $c_3=0.3$, and $c_4=0.4$.
Each instance of our simulation uses $4$ Hurwitz matrices that are generated randomly beforehand.
Finally, we used the maximum eigenvalue of the random matrices as the maximum rate of changes ($\max\Brace{\dot{y}_i(t)\mid t\in[0,T]}$) in our algorithm.

Our implementation is in Scala.
We used the Apache Commons Mathematics Library~\cite{Apache} to find eigenvalues of a matrix.
Our simulations are performed on Ubuntu 18.04 with i7-8700 CPU 3.2GHz and 16GB memory.
We ran each test $50$ times and 
report average running time as well as the 95\% confidence intervals.
Figure~\ref{fig:bounded} shows the results for the case that $T$ is bounded ($1 000$ and $10 000$), and
Figure~\ref{fig:unbounded} shows the results for the case that $T$ is set to $\infty$.
`Threshold' is the value of $p$ in our desired property.
`\#states' is the number of states in CTMC.
`\#checks' is the number of checkpoints the algorithm uses to discretize the time. 
This number does not tell how many steps the algorithm takes to simulate the system for $T$ units of time (or until it reaches the invariant distribution). 
It is the number of points in time, for which we examine the distribution of the state.
When the time is unbounded (\ie\ $T=\infty$ in Figure~\ref{fig:unbounded}), the algorithm first finds a time when the system sufficiently convergences to the invariant distribution.
It is easy to see that in the invariant distribution, our example is reduced to a birth–death process, 
for which we can compute the invariant distribution analytically.
Figure~\ref{fig:infH} shows the average amount of time our algorithm spent to find a time in which the distribution is known to be invariant.
Figure~\ref{fig:infR} shows the average amount of time the algorithm uses to verify the property after a time horizon is fixed (note that our property of interest does not hold at the invariant distribution). 
Figure~\ref{fig:infT} shows the sum of previous averages.

As expected, the time consumption of our algorithm increases logarithmically with the number of the states.
This is because in statistical model checking, the number of required samples is independent of the number of the states, and the time to draw each sample grows logarithmically with the number of the states.

\section{Conclusion} \label{sec:conclusion}

In this work, we proposed a method of verifying temporal logic formulas on stochastic hybrid systems via model reduction in both continuous-time and discrete-time.
Specifically, we reduce stochastic hybrid systems to Markov chains by partitioning the state space. We present an upper bound on the error introduced due to this reduction.
In addition, we present stochastic algorithms that verify temporal logic formulas on Markov chains with arbitrarily high confidence.


\begin{thebibliography}{10}
  \providecommand{\url}[1]{#1}
  \csname url@samestyle\endcsname
  \providecommand{\newblock}{\relax}
  \providecommand{\bibinfo}[2]{#2}
  \providecommand{\BIBentrySTDinterwordspacing}{\spaceskip=0pt\relax}
  \providecommand{\BIBentryALTinterwordstretchfactor}{4}
  \providecommand{\BIBentryALTinterwordspacing}{\spaceskip=\fontdimen2\font plus
  \BIBentryALTinterwordstretchfactor\fontdimen3\font minus
    \fontdimen4\font\relax}
  \providecommand{\BIBforeignlanguage}[2]{{%
  \expandafter\ifx\csname l@#1\endcsname\relax
  \typeout{** WARNING: IEEEtran.bst: No hyphenation pattern has been}%
  \typeout{** loaded for the language `#1'. Using the pattern for}%
  \typeout{** the default language instead.}%
  \else
  \language=\csname l@#1\endcsname
  \fi
  #2}}
  \providecommand{\BIBdecl}{\relax}
  \BIBdecl
  
  \bibitem{jin2014benchmarks}
  X.~Jin, J.~V. Deshmukh, J.~Kapinski, K.~Ueda, and K.~Butts, ``Benchmarks for
    model transformations and conformance checking,'' in \emph{1st International
    Workshop on Applied Verification for Continuous and Hybrid Systems (ARCH)},
    2014.
  
  \bibitem{daniele2017smart}
  I.~Daniele, F.~Alessandro, H.~Marianne, B.~Axel, and P.~Maria, ``A smart grid
    energy management problem for data-driven design with probabilistic
    reachability guarantees,'' in \emph{4th International Workshop on Applied
    Verification of Continuous and Hybrid Systems}, 2017, pp. 2--19.
  
  \bibitem{rajkumar2010cyber}
  R.~R. Rajkumar, I.~Lee, L.~Sha, and J.~Stankovic, ``Cyber-physical systems: the
    next computing revolution,'' in \emph{Proceedings of the 47th design
    automation conference}.\hskip 1em plus 0.5em minus 0.4em\relax ACM, 2010, pp.
    731--736.
  
  \bibitem{liu_probabilistic_2011}
  B.~Liu, D.~Hsu, and P.~S. Thiagarajan, ``Probabilistic approximations of {ODEs}
    based bio-pathway dynamics,'' \emph{Theoretical Computer Science}, vol. 412,
    no.~21, pp. 2188--2206, May 2011.
  
  \bibitem{liu_approximate_2012}
  B.~Liu, A.~Hagiescu, S.~K. Palaniappan, B.~Chattopadhyay, Z.~Cui, W.-F. Wong,
    and P.~S. Thiagarajan, ``Approximate probabilistic analysis of biopathway
    dynamics,'' \emph{Bioinformatics}, vol.~28, no.~11, pp. 1508--1516, Jun.
    2012.
  
  \bibitem{zuliani_statistical_2014}
  P.~Zuliani, ``Statistical model checking for biological applications,''
    \emph{STTT}, pp. 1--10, Aug. 2014.
  
  \bibitem{gyori2015approximate}
  B.~M. Gyori, B.~Liu, S.~Paul, R.~Ramanathan, and P.~Thiagarajan, ``Approximate
    probabilistic verification of hybrid systems,'' in \emph{Hybrid Systems
    Biology}.\hskip 1em plus 0.5em minus 0.4em\relax Springer, 2015, pp. 96--116.
  
  \bibitem{decidable-hybrid98}
  T.~Henzinger, P.~Kopke, A.~Puri, and P.~Varaiya, ``What's decidable about
    hybrid automata?'' \emph{Journal of Computer and System Sciences}, vol.~57,
    no.~1, pp. 94--124, 1998.
  
  \bibitem{efhkost03-2}
  E.~Clarke, A.~Fehnker, Z.~Han, B.~Krogh, J.~Ouaknine, O.~Stursberg, and
    M.~Theobald, ``Abstraction and {C}ounterexample-{G}uided {R}efinement in
    {M}odel {C}hecking of {H}ybrid {S}ystems,'' \emph{JFCS}, vol.~14, no.~4, pp.
    583--604, 2003.
  
  \bibitem{adi03}
  R.~Alur, T.~Dang, and F.~Ivancic, ``Counter-{E}xample {G}uided {P}redicate
    {A}bstraction of {H}ybrid {S}ystems,'' in \emph{TACAS 2003}, 2003, pp.
    208--223.
  
  \bibitem{rpv17}
  N.~Roohi, P.~Prabhakar, and M.~Viswanathan, ``{HARE}: {A} {Hybrid}
    {Abstraction} {Refinement} {Engine} for verifying non-linear hybrid
    automata,'' in \emph{Proceedings of TACAS}, 2017, pp. 573--588.
  
  \bibitem{tabuada_linear_2006}
  P.~Tabuada and G.~Pappas, ``Linear time logic control of discrete-time linear
    systems,'' \emph{{IEEE} Transactions on Automatic Control}, vol.~51, no.~12,
    pp. 1862--1877, Dec. 2006.
  
  \bibitem{kloetzer_fully_2008}
  M.~Kloetzer and C.~Belta, ``A fully automated framework for control of linear
    systems from temporal logic specifications,'' \emph{{IEEE} Transactions on
    Automatic Control}, vol.~53, no.~1, pp. 287--297, Feb. 2008.
  
  \bibitem{wongpiromsarn_receding_2010}
  T.~Wongpiromsarn, U.~Topcu, and R.~M. Murray, ``Receding horizon control for
    temporal logic specifications,'' in \emph{Proceedings of the 13th {ACM}
    International Conference on Hybrid Systems: Computation and Control}, ser.
    {HSCC} '10.\hskip 1em plus 0.5em minus 0.4em\relax New York, {NY}, {USA}:
    {ACM}, 2010, pp. 101--110.
  
  \bibitem{liu2013synthesis}
  J.~Liu, N.~Ozay, U.~Topcu, and R.~M. Murray, ``Synthesis of reactive switching
    protocols from temporal logic specifications,'' \emph{IEEE Transactions on
    Automatic Control}, vol.~58, no.~7, pp. 1771--1785, 2013.
  
  \bibitem{cv09-2}
  R.~Chadha and M.~Viswanathan, ``A {C}ounterexample {G}uided
    {A}bstraction-{R}efinement {F}ramework for {M}arkov {D}ecision {P}rocesses,''
    \emph{ACM Transactions on Computational Logic}, vol.~12, no.~1, pp.
    1:1--1:49, 2010.
  
  \bibitem{tkachev2013formula}
  I.~Tkachev and A.~Abate, ``Formula-free finite abstractions for linear temporal
    verification of stochastic hybrid systems,'' in \emph{Proceedings of the 16th
    international conference on Hybrid Systems: Computation and Control}.\hskip
    1em plus 0.5em minus 0.4em\relax ACM, 2013, pp. 283--292.
  
  \bibitem{tkachev2013quantitative}
  I.~Tkachev, A.~Mereacre, J.-P. Katoen, and A.~Abate, ``Quantitative
    automata-based controller synthesis for non-autonomous stochastic hybrid
    systems,'' in \emph{Proceedings of the 16th international conference on
    Hybrid Systems: Computation and Control}.\hskip 1em plus 0.5em minus
    0.4em\relax ACM, 2013, pp. 293--302.
  
  \bibitem{04-iLTL}
  Y.~Kwon and G.~Agha, ``Linear inequality ltl (iltl): A model checker for
    discrete time markov chains,'' in \emph{Formal Methods and Software
    Engineering}, ser. Lecture Notes in Computer Science, J.~Davies, W.~Schulte,
    and M.~Barnett, Eds.\hskip 1em plus 0.5em minus 0.4em\relax Springer Berlin
    Heidelberg, 2004, vol. 3308, pp. 194--208.
  
  \bibitem{96-MITL}
  R.~Alur, T.~Feder, and T.~A. Henzinger, ``The benefits of relaxing
    punctuality,'' \emph{J. ACM}, vol.~43, no.~1, pp. 116--146, 1996.

  \bibitem{roohi2018revisiting}
  N.~Roohi and M.~Viswanathan,
  ``Revisiting MITL to fix decision procedures,'' 
  \emph{International Conference on Verification, Model Checking, and Abstract Interpretation}, pp. 474--494, 2018.
  
  \bibitem{chorin_optimal_2000}
  A.~J. Chorin, O.~H. Hald, and R.~Kupferman, ``Optimal prediction and the
    mori-zwanzig representation of irreversible processes,'' \emph{Proceedings of
    the National Academy of Sciences}, vol.~97, no.~7, pp. 2968--2973, Mar. 2000.
  
  \bibitem{beck_model_2009}
  C.~Beck, S.~Lall, T.~Liang, and M.~West, ``Model reduction, optimal prediction,
    and the mori-zwanzig representation of markov chains,'' in
    \emph{{CDC}/{CCC}}, 2009, pp. 3282--3287.
  
  \bibitem{julius_ApproximationsStochasticHybrid_2009}
  A.~A. Julius and G.~J. Pappas, ``Approximations of {{Stochastic Hybrid
    Systems}},'' \emph{IEEE Transactions on Automatic Control}, vol.~54, no.~6,
    pp. 1193--1203, 2009.
  
  \bibitem{abate_ApproximateModelChecking_2010a}
  A.~Abate, J.-P. Katoen, J.~Lygeros, and M.~Prandini,
    ``\BIBforeignlanguage{en}{Approximate {{Model Checking}} of {{Stochastic
    Hybrid Systems}}},'' \emph{\BIBforeignlanguage{en}{European Journal of
    Control}}, vol.~16, no.~6, pp. 624--641, 2010.
  
  \bibitem{abate_ApproximateAbstractionsStochastic_2011}
  A.~Abate, A.~D'Innocenzo, and M.~D.~D. Benedetto, ``Approximate
    {{Abstractions}} of {{Stochastic Hybrid Systems}},'' \emph{IEEE Transactions
    on Automatic Control}, vol.~56, no.~11, pp. 2688--2694, 2011.
  
  \bibitem{Pola20082508}
  G.~Pola, A.~Girard, and P.~Tabuada, ``Approximately bisimilar symbolic models
    for nonlinear control systems,'' \emph{Automatica}, vol.~44, no.~10, pp. 2508
    -- 2516, 2008.
  
  \bibitem{girard2010approximately}
  A.~Girard, G.~Pola, and P.~Tabuada, ``Approximately bisimilar symbolic models
    for incrementally stable switched systems,'' \emph{IEEE Transactions on
    Automatic Control}, vol.~55, no.~1, pp. 116--126, 2010.
  
  \bibitem{zamani2012symbolic}
  M.~Zamani, G.~Pola, M.~Mazo, and P.~Tabuada, ``Symbolic models for nonlinear
    control systems without stability assumptions,'' \emph{IEEE Transactions on
    Automatic Control}, vol.~57, no.~7, pp. 1804--1809, 2012.
  
  \bibitem{zamani2014symbolic}
  M.~Zamani, P.~M. Esfahani, R.~Majumdar, A.~Abate, and J.~Lygeros, ``Symbolic
    control of stochastic systems via approximately bisimilar finite
    abstractions,'' \emph{IEEE Transactions on Automatic Control}, vol.~59,
    no.~12, pp. 3135--3150, 2014.
  
  \bibitem{younes_statistical_2006}
  H.~L.~S. Younes and R.~G. Simmons, ``Statistical probabilistic model checking
    with a focus on time-bounded properties,'' \emph{Information and
    Computation}, vol. 204, no.~9, pp. 1368--1409, Sep. 2006.
  
  \bibitem{sen_statistical_2005}
  K.~Sen, M.~Viswanathan, and G.~Agha, ``On statistical model checking of
    stochastic systems,'' in \emph{Computer Aided Verification}, ser. Lecture
    Notes in Computer Science, K.~Etessami and S.~K. Rajamani, Eds.\hskip 1em
    plus 0.5em minus 0.4em\relax Springer Berlin Heidelberg, Jan. 2005, no. 3576,
    pp. 266--280.
  
  \bibitem{yesno-to-yesnounknown-2006-Younes}
  H.~L.~S. Younes, ``Error control for probabilistic model checking,'' in
    \emph{Verification, Model Checking, and Abstract Interpretation, 7th
    International Conference, {VMCAI} 2006, Charleston, SC, USA, January 8-10,
    2006, Proceedings}, 2006, pp. 142--156.
  
  \bibitem{rwwdv17}
  N.~Roohi, Y.~Wang, M.~West, G.~Dullerud, and M.~Viswanathan, ``Statistical
    verification of the {Toyota} powertrain control verification benchmark,'' in
    \emph{Proceedings of HSCC}, 2017, pp. 65--70.
  
  
  \bibitem{wang2015statistical}
  Y.~Wang, N.~Roohi, M.~West, M.~Viswanathan, and G.~E. Dullerud, ``Statistical
    verification of dynamical systems using set oriented methods,'' in
    \emph{Proceedings of the 18th International Conference on Hybrid Systems:
    Computation and Control}.\hskip 1em plus 0.5em minus 0.4em\relax ACM, 2015,
    pp. 169--178.
  
  \bibitem{Wang2015267}
  ------, ``A {Mori-Zwanzig} and {MITL} based approach to statistical
    verification of continuous-time dynamical systems,''
    \emph{IFAC-PapersOnLine}, vol.~48, no.~27, pp. 267--273, 2015.
  
  \bibitem{wang2016verifying}
  ------, ``Verifying continuous-time stochastic hybrid systems via mori-zwanzig
    model reduction,'' in \emph{Decision and Control (CDC), 2016 IEEE 55th
    Conference on}.\hskip 1em plus 0.5em minus 0.4em\relax IEEE, 2016, pp.
    3012--3017.
  
  
  \bibitem{Teel20142435}
  A.~R. Teel, A.~Subbaraman, and A.~Sferlazza, ``Stability analysis for
    stochastic hybrid systems: A survey,'' \emph{Automatica}, vol.~50, no.~10,
    pp. 2435 -- 2456, 2014.
  
  \bibitem{teel2015stochastic}
  A.~R. Teel and J.~P. Hespanha, ``Stochastic hybrid systems: a modeling and
    stability theory tutorial,'' in \emph{Decision and Control (CDC), 2015 IEEE
    54th Annual Conference on}.\hskip 1em plus 0.5em minus 0.4em\relax IEEE,
    2015, pp. 3116--3136.
  
  \bibitem{Teel2017}
  A.~R. Teel, \emph{Recent Developments in Stability Theory for Stochastic Hybrid
    Inclusions}.\hskip 1em plus 0.5em minus 0.4em\relax Cham: Springer
    International Publishing, 2017, pp. 329--354.
  
  \bibitem{Subbaraman2017}
  A.~Subbaraman and A.~R. Teel, ``Robust global recurrence for a class of
    stochastic hybrid systems,'' \emph{Nonlinear Analysis: Hybrid Systems},
    vol.~25, pp. 283 -- 297, 2017.
  
  \bibitem{karatzas2012brownian}
  I.~Karatzas and S.~Shreve, \emph{Brownian motion and stochastic
    calculus}.\hskip 1em plus 0.5em minus 0.4em\relax Springer Science \&
    Business Media, 2012, vol. 113.
  
  \bibitem{revuz2013continuous}
  D.~Revuz and M.~Yor, \emph{Continuous martingales and Brownian motion}.\hskip
    1em plus 0.5em minus 0.4em\relax Springer Science \& Business Media, 2013,
    vol. 293.
  
  \bibitem{hanson_AppliedStochasticProcesses_2007}
  F.~B. Hanson, ``\BIBforeignlanguage{en}{Applied {{Stochastic Processes}} and
    {{Control}} for {{Jump}}-{{Diffusions}}: {{Modeling}}, {{Analysis}} and
    {{Computation}}},'' p.~29, 2007.
  
  \bibitem{04-STL}
  O.~Maler and D.~Nickovic, \emph{Monitoring Temporal Properties of Continuous
    Signals}, 2004, pp. 152--166.
  
  \bibitem{15-monSTL}
  J.~V. Deshmukh, A.~Donz{\'e}, S.~Ghosh, X.~Jin, G.~Juniwal, and S.~A. Seshia,
    \emph{Robust Online Monitoring of Signal Temporal Logic}, 2015, pp. 55--70.
  
  \bibitem{13-monSTL}
  A.~Donz{\'e}, T.~Ferr{\`e}re, and O.~Maler, \emph{Efficient Robust Monitoring
    for STL}, 2013, pp. 264--279.
  
  \bibitem{94-TA}
  R.~Alur and D.~L. Dill, ``A theory of timed automata,'' \emph{Theor. Comput.
    Sci.}, vol. 126, no.~2, pp. 183--235, Apr. 1994.
  
  \bibitem{rudin1973functional}
  W.~Rudin, ``Functional analysis,'' 1973.
  
  \bibitem{dist-closeness-2013-Batu}
  T.~Batu, L.~Fortnow, R.~Rubinfeld, W.~D. Smith, and P.~White, ``Testing
    closeness of discrete distributions,'' \emph{J. ACM}, vol.~60, no.~1, pp.
    4:1--4:25, Feb. 2013.
  
  \bibitem{gillespie1976general}
  D.~T. Gillespie, ``A general method for numerically simulating the stochastic
    time evolution of coupled chemical reactions,'' \emph{Journal of
    computational physics}, vol.~22, no.~4, pp. 403--434, 1976.
  
  \bibitem{sprt-1945-wald}
  A.~Wald, ``Sequential tests of statistical hypotheses,'' \emph{The Annals of
    Mathematical Statistics}, vol.~16, no.~2, pp. pp. 117--186, 1945.
  
  \bibitem{chow-robbins-1965}
  Y.~S. Chow and H.~Robbins, ``On the asymptotic theory of fixed-width sequential
    confidence intervals for the mean,'' \emph{The Annals of Mathematical
    Statistics}, vol.~36, no.~2, pp. 457--462, 04 1965.
  
  \bibitem{ltl2buchi2001Gastin}
  P.~Gastin and D.~Oddoux, ``Fast ltl to b\"uchi automata translation,'' in
    \emph{Proceedings of the 13th International Conference on Computer Aided
    Verification}, ser. CAV '01.\hskip 1em plus 0.5em minus 0.4em\relax London,
    UK, UK: Springer-Verlag, 2001, pp. 53--65.
  
  \bibitem{ltl-spot-2011-Duret}
  A.~Duret-Lutz, ``Ltl translation improvements in spot,'' in \emph{Proceedings
    of the Fifth International Conference on Verification and Evaluation of
    Computer and Communication Systems}, ser. VECoS'11.\hskip 1em plus 0.5em
    minus 0.4em\relax Swinton, UK, UK: British Computer Society, 2011, pp.
    72--83.
  
  \bibitem{spot-2004-Duret}
  A.~Duret-Lutz and D.~Poitrenaud, ``Spot: an extensible model checking library
    using transition-based generalized b\"uchi automata,'' in \emph{IN PROC. OF
    MASCOTS'04}.\hskip 1em plus 0.5em minus 0.4em\relax IEEE Computer Society,
    2004, pp. 76--83.
  
  \bibitem{dellnitz_approximation_1999}
  M.~Dellnitz and O.~Junge, ``On the approximation of complicated dynamical
    behavior,'' \emph{{SIAM} Journal on Numerical Analysis}, vol.~36, no.~2, pp.
    491--515, Jan. 1999.
  
  \bibitem{buchi}
  \BIBentryALTinterwordspacing
  A.~P. Sistla and E.~M. Clarke, ``The complexity of propositional linear
    temporal logics,'' \emph{J. ACM}, vol.~32, no.~3, pp. 733--749, Jul. 1985.
    [Online]. Available: \url{http://doi.acm.org/10.1145/3828.3837}
  \BIBentrySTDinterwordspacing
  
  \bibitem{Apache}
  ``{Commons Math: The Apache Commons Mathematics Library},''
    \url{https://commons.apache.org/proper/commons-math}, accessed: 2019-06-10.
  
  \end{thebibliography}
\end{document}